\documentclass[11pt]{amsart}

\usepackage[dvipsnames]{xcolor}

\usepackage{CJKutf8}

\usepackage{ytableau,tikz,varwidth, hyperref, amssymb, mathscinet,mathtools, amsthm}
\usetikzlibrary{calc, decorations}
\usepackage[enableskew]{youngtab}
\usepackage[left=1in,right=1in, top=1in, bottom=1.25in]{geometry}
\usepackage{enumerate, multirow}
\usepackage[shortlabels]{enumitem}
\usepackage[all,cmtip]{xy}
\usepackage{caption}
\usepackage{diagbox}
\usepackage{subcaption}
\usepackage{tikz-cd}
\usepackage{mathrsfs}
\usepackage[normalem]{ulem} 

\newtheorem{thm}{Theorem}[section]

\usepackage[dvipsnames]{xcolor}

\usepackage[dvipsnames]{xcolor}

\theoremstyle{definition} 

\newtheorem{notation}[thm]{Notation}

\newcommand{\Agtropm}{A_g[m]^\trop}

\newcommand{\PD}{\mathrm{PD}}
\newcommand{\PDrt}{\mathrm{PD}^{\mathrm{rt}}}

\newcommand{\ov}{\overline}

\newcommand{\col}{\colon}

\newcommand{\hide}[1]{}
\newcommand{\ans}[1]{}

\usepackage[thinlines]{easytable}

\newcommand{\down}[2]{\xymatrix@R=6mm@C=2mm{
#1\ar[d]\\ #2
}}

\newcommand{\downlabel}[3]{\xymatrix@R=6mm@C=2mm{
{#1}\ar[d]^<<<{#3} \\ #2
}}
\newcommand{\squarediagrammapsto}[4]{\xymatrix@R=8mm@C=8mm{
#1\ar@{|->}[d]\ar@{|->}[r] & #2\ar@{|->}[d] \\ #3\ar@{|->}[r] &#4
}}
\newcommand{\squarediagramlabel}[8]{\xymatrix@R=8mm@C=8mm{
#1\ar[d]_{#6}\ar[r]^{#5} & #2\ar[d]^{#7} \\ #3\ar[r]^{#8} &#4
}}

\newcommand{\isocelesdown}[3]{\xymatrix@R=6mm@C=0mm{
& {#1}\ar[dl] \ar[dr] & \\
{#2} \ar[rr] && {#3}
}}

\newcommand{\isocelesdownlabel}[6]{\xymatrix@R=6mm@C=0mm{
& {#1}\ar[dl]_<<<<{#4} \ar[dr]^<<<<{#5} & \\
{#2} \ar[rr]_{#6} && {#3}
}}

\newcommand{\isocelesup}[3]{\xymatrix@R=6mm@C=0mm{
#1\ar[rr]\ar[dr] && #2\ar[dl] \\
& #3 &
}}

\newcommand{\isocelesuplabel}[6]{\xymatrix@R=6mm@C=0mm{
#1\ar[rr]^{{#4}} \ar[dr]_<<<{#5} && #2\ar[dl]^<<<{#6} \\
& #3 &
}}

\newtheorem{Definition}[thm]{Definition}
\newenvironment{definition}
{\begin{Definition}\rm}{\end{Definition}}

\newtheorem{Example}[thm]{Example}
\newenvironment{example}
{\begin{Example}\rm}{\end{Example}}

\newtheorem{Exercise}[thm]{Exercise}
\newenvironment{exercise}
{\begin{Exercise}\rm}{\end{Exercise}}

\newtheorem{Fact}[thm]{Fact}
\newenvironment{fact}
{\begin{Fact}\rm}{\end{Fact}}

\newtheorem{Theorem}[thm]{Theorem}
\newenvironment{theorem}
{\begin{Theorem}\rm}{\end{Theorem}}

\newtheorem{Lemma}[thm]{Lemma}
\newenvironment{lemma}
{\begin{Lemma}\rm}{\end{Lemma}}

\newtheorem{Remark}[thm]{Remark}
\newenvironment{remark}
{\begin{Remark}\rm}{\end{Remark}}

\newtheorem{Proposition}[thm]{Proposition}
\newenvironment{proposition}
{\begin{Proposition}\rm}{\end{Proposition}}

\newtheorem{Corollary}[thm]{Corollary}
\newenvironment{corollary}
{\begin{Corollary}\rm}{\end{Corollary}}

\newtheorem{Question}[thm]{Question}
\newenvironment{question}
{\begin{Question}\rm}{\end{Question}}

\newtheorem{Conjecture}[thm]{Conjecture}
\newenvironment{conjecture}
{\begin{Conjecture}\rm}{\end{Conjecture}}

\newtheorem{Problem}[thm]{Problem}

\newtheorem{Hypothesis}[thm]{Hypothesis}

\newtheorem{manualtheoreminner}{Theorem}
\newenvironment{manualtheorem}[1]{%
  \IfBlankTF{#1}
    {\renewcommand{\themanualtheoreminner}{\unskip}}
    {\renewcommand\themanualtheoreminner{#1}}%
  \manualtheoreminner
}{\endmanualtheoreminner}

\theoremstyle{remark}

\newcommand \enumnow[1]{\begin{enumerate}{#1}\end{enumerate}}

\newcommand{\on}{\operatorname}

\newcommand{\Aut}{\on{Aut}}

\newcommand{\GL}{\mathrm{GL}}
\newcommand{\Gr}{\operatorname{Gr}}

\DeclareMathOperator{\id}{id}

\newcommand{\Mat}{\on{Mat}}

\DeclareMathOperator{\rank}{rank}
\newcommand{\rat}{\mathrm{rt}}

\DeclareMathOperator{\rt}{rt}

\newcommand{\SL}{\mathrm{SL}}

\newcommand{\Sp}{\mathrm{Sp}}

\DeclareMathOperator{\St}{St}

\newcommand{\Sym}{\mathrm{Sym}}

\newcommand{\trop}{\mathrm{trop}}

\makeatletter
\newcommand*{\tp}{%
  {\mathpalette\@transpose{}}%
}
\newcommand*{\@transpose}[2]{%
  \raisebox{\depth}{$\m@th#1\intercal$}%
}
\makeatother

\newcommand{\FF}{\mathbb{F}}

\newcommand{\HH}{\mathbb{H}}

\newcommand{\PP}{\mathbb{P}}
\newcommand{\Q}{\mathbb{Q}}
\newcommand{\QQ}{\mathbb{Q}}
\newcommand{\R}{\mathbb{R}}
\newcommand{\RR}{\mathbb{R}}

\newcommand{\Z}{\mathbb{Z}}
\newcommand{\ZZ}{\mathbb{Z}}

\newcommand{\cA}{\mathcal{A}}

\newcommand{\cM}{\mathcal{M}}

\newcommand{\cP}{\mathcal{P}}

\newcommand{\cT}{\mathcal{T}}

\newcommand{\cW}{\mathcal{W}}

\usepackage{array}
\newcolumntype{M}[1]{>{\centering\arraybackslash}m{#1}}
\newcolumntype{N}{@{}m{0pt}@{}}

\title{Level structures on tropical 
abelian varieties}

\author[E. Assaf]{Eran Assaf}
\address{Department of Mathematics, Massachusetts Institute of Technology, Cambridge, MA 02139}
\email{eranasaf@mit.edu}

\author[M. Brandt]{Madeline Brandt}
 \address{Department of Mathematics, Vanderbilt University, Nashville, TN 37240}
\email{madeline.v.brandt@vanderbilt.edu}

\author[J. Bruce]{Juliette Bruce}
\address{Department of Mathematics, Dartmouth College, Hanover, NH 03755}
\email{juliette.bruce@dartmouth.edu}

\author[M. Chan]{Melody Chan}
\address{Department of Mathematics, Brown University, Providence, RI 02912}
\email{melody\_chan@brown.edu}

\author[R. Vlad]{Raluca Vlad}
\address{Department of Mathematics, Brown University, Providence, RI 02912}
\email{raluca\_vlad@brown.edu}

\begin{document}

\begin{abstract}
We introduce level structures on tropical abelian varieties and give a modular interpretation of $A_g[m]^\trop$, the tropicalization of the moduli space $\cA_g[m]$ of principally polarized abelian varieties with level $m$ structure. We study the case of abelian surfaces in greater depth. The link of $A_2[m]^\trop$ is an explicit simplicial complex whose vertices are the primitive vectors of $(\ZZ/m\ZZ)^4$ up to sign. As a topological space, this link is homotopic to a wedge sum of closed orientable surfaces and circles; we compute the number of each of these and the genera of the surfaces. We  deduce the weight zero compactly supported rational cohomology of $\cA_2[m]$, completing a calculation of Oda--Schwermer from 1990.
\end{abstract}

\maketitle

\section{Introduction}

We study the tropicalization $A_g[m]^\trop$ of the moduli space of principally polarized abelian varieties of dimension $g$ with level structure. The definition of $A_g[m]^\trop$ relies on the theory of tropicalizations of toroidal compactifications of locally symmetric varieties which were studied in great generality in~\cite{ABBCV} following~\cite{amrt, kkmsd}. Specifically, $A_g[m]^\trop$ is a topological space homeomorphic to the cone over the dual complex of the boundary of \emph{any} toroidal compactification of the moduli space $\cA_g[m]$ of principally polarized $g$-dimensional abelian varieties with level structure~$m$. 

We introduce the notion of level structures on principally polarized tropical abelian varieties, which leads to a tropical modular interpretation for $A_g[m]^\trop$ in Section~\ref{subsection-tropical-modular-interpretation}. This result is in the spirit of \cite{acp}, where the boundary complex of the Deligne--Mumford compactification $\cM_{g,n}\subset \ov \cM_{g,n}$ of the moduli of marked algebraic curves is shown to be a tropical moduli space of graphs.
Our result is a generalization of \cite{bmv}, which interprets $A_g^{\trop}$ as a moduli space of principally polarized tropical abelian varieties \cite{MZ08}. As one would expect, our definition recovers the original definition when $m=1$ (Remark~\ref{rem:same as bmv}).

Following~\cite[Section 4.1]{ABBCV}, we give a completely linear algebraic definition of $A_g[m]^\trop$. The space $A_g[m]^\trop$ is built out of polyhedral fan decompositions of the rational closures $\PD(W^\vee)^{\rt}$ of the cones of positive definite forms on $W^\vee_\R$, as $W$ ranges over all isotropic subspaces of $\Z^{2g}$.

These polyhedral fans are compatible under identifications induced from the congruence subgroup $\Gamma[m] = \Sp_{2g}(\Z)[m]$ of the integral symplectic group -- i.e., under all natural inclusions 
$$\PD(W^\vee)^{\rt} \hookrightarrow \PD((W')^\vee)^{\rt}$$ 
when $\gamma \cdot W \subseteq W'$ for $\gamma \in \Gamma[m]$.
This is the notion of an admissible collection, which we recall in Section~\ref{subsec:def-Agmtrop}, together with the formal definition of $A_g[m]^\trop$ as the geometric realization of a gluing of these fans.

A choice of basis $W \cong \Z^t$ induces an identification of $\PD(W^\vee)^{\rt}$ with the cone $\PDrt_t$ of positive semi-definite forms on $\R^t$ with kernel defined over $\Q$. There are three classical examples of polyhedral decompositions for $\PDrt_t$ for all $t$, called the perfect cone or first Voronoi, the second Voronoi, and the central cone decompositions. Fixing any of these three families of examples, the induced decompositions of $\PD(W^\vee)^{\rt}$ for all isotropic subspaces $W$ form an admissible collection and give a cellular structure on $A_g[m]^\trop$. All three families coincide when $t\le 2$ \cite[Example 4.3.1]{bmv}, giving rise to the following simplicial complex structure on the link of $A_2[m]^\trop$.

\begin{manualtheorem} {\ref{thm:simplicial-complex-description}} 
Fix $m \geq 2$. The link $LA_2[m]^\trop$ is isomorphic, via the ray labeling from Section~\ref{sec:gluing-pattern}, to the $2$-dimensional simplicial complex $\Delta_2[m]$ defined as follows.
\begin{enumerate}
    \item[(a)] The vertex set is the set of primitive elements in $(\Z/m\Z)^4/\{\pm 1\}$.
    \item[(b)] For $v \in (\ZZ/m\ZZ)^4$, write $[v]$ for its equivalence class up to sign. Then two vertices $[v]$ and $[w]$ form an edge if and only if $\{v,w\}$ is a $(\ZZ/m\ZZ)$-basis for an isotropic plane in $(\Z/m\Z)^4$.
    \item[(c)] Three vertices $[u],$ $[v]$, and $[w]$ form a 2-simplex if and only if each pair forms an edge and there exists a choice of signs such that \[\pm u \pm v \pm w = 0.\]
\end{enumerate}
\end{manualtheorem}

See Figure~\ref{fig:a22} for a depiction of this simplicial structure in the case $m = 2$, and Corollary~\ref{cor:counts-simplices-A2m} for a count of the cells when $m \geq 3$.

\medskip

The moduli space $\cA_g[m]$ is a level cover of the moduli $\cA_g$ of principally polarized abelian varieties. In analogy to the classical case, the tropicalization $A_g[m]^\trop$ admits a natural surjective map to $A_g^\trop$. Concretely, $A_g^\trop = \PD_g^{\rt}/\GL_g(\Z)$, while $A_g[m]^\trop$ is a connected topological space constructed from a finite number of copies of the quotient of $\PD_g^{\rt}$ by the congruence subgroup
$$\GL_g(\Z)[m] \; := \; \ker\big( \GL_g(\Z) \xrightarrow{\mathrm{mod} \; m} \GL_g(\Z/m\Z)\big).$$

When $g = 2$, we understand these quotient spaces and their gluings sufficiently well in order to describe the homotopy type of $A_2[m]^\trop$. For $m \geq 3$, we show that the link of $\PDrt_2/\GL_2(\Z)[m]$, which coincides with the compact modular curve $X(m)$, is homeomorphic to a closed orientable surface $\Sigma_{h_m}$ of genus 
\[h_m \; = \; 1 - \frac{m^2(6-m)}{24} \cdot \prod \nolimits_{p \mid m} (1- p^{-2}).\]

We also prove that the link $LA_2[m]^\trop$ is built out of $\pi_m = \frac{m^4}{2} \cdot \prod_{p\mid m} (1-p^{-4})$ copies of $\Sigma_{h_m}$; namely $\pi_{m}$ compact modular curves $X(m)$ glued together at cusps.
Moreover, these surfaces are glued according to the incidence structure of non-zero isotropic subspaces of $\Z^4$, modulo the action by $\Gamma[m]$. Concretely, we define the graph $G_m$ to be the quotient of the rational symplectic Tits building $\cT^\omega_4(\Q)$ (whose vertices correspond to isotropic lines and planes of $\Q^4$ and edges are given by inclusions) by the natural action by $\Gamma[m]$. Setting $c_{1,m} = \frac{m^2}{2} \cdot \prod_{p \mid m} (1-p^{-2})$, we obtain the following.

\begin{manualtheorem} {\ref{thm:homotopy}}
    For all $m\geq3$, there is a homotopy equivalence:
    \begin{equation} \label{eq:homotopy-intro}
        LA_{2}[m]^{\trop} \;\; \simeq  \;\; \left(\bigvee_{i=1}^{\pi_m} \Sigma_{h_m} \right) \vee \left( \bigvee_{j=1}^{\beta_{1}(G_m)} S^1\right),
    \end{equation}
    where $\beta_1(G_m)=\pi_m(c_{1,m}-2) + 1$ is the first Betti number of the graph $G_m$.
\end{manualtheorem}

A consequence of this homotopy equivalence is that, when $m$ is prime, the $\pm$-oriented symplectic Steinberg module studied in~\cite{capovillaSearle26} is a direct summand of the homology in degree 1 of $LA_2[m]^\trop$, see Remark~\ref{rem:steinberg}.

Our descriptions of the topology of $A_2[m]^\trop$ sit alongside a number of recent results in the literature determining the homeomorphism or simple homotopy type of boundary complexes of interesting compactified moduli spaces. See, for example, \cite{allcock-corey-payne-tropical}, \cite{el-maazouz-helminck-roehrle-souza-yun-topology}, and \cite{kannan-song-dual}.

By~\cite[Theorem~1.19]{ABBCV}, the tropicalization $A_g[m]^\trop$ captures the weight zero compactly supported cohomology of $\cA_g[m]$, in the sense that we have a natural isomorphism
$$W_0 H_c^\ast(\cA_g[m]; \, \Q) \; \cong \; H^*_c(A_g[m]^\trop; \, \Q).$$
Using our results about the topology of $A_2[m]^\trop$, we subsequently compute the weight zero compactly supported cohomology of $\cA_2[m]$ in Corollary~\ref{cor:A22} and Theorem~\ref{thm:a2m}.
Oda--Schwermer study the weight filtration on $H^*(\cA_2[m];\QQ)$ and achieve a near-complete calculation of the top-weight piece; our results, in particular, complete their calculation, see Remark~\ref{rem:oda-schwermer}.

\medskip

\noindent \textbf{Future directions.} It would be an interesting direction of future work to find analogous explicit descriptions of the homotopy type of  $LA_g[m]^\trop$ for $g \geq 3$. We expect there exists no description which is universal for all $g$ and allows for a full computation of the weight zero compactly supported cohomology of $\cA_{g}[m]$. However, it would be interesting to see if such computations are possible for small values of $g$, similar to the calculations done in \cite{bbcmmw-top} for $\cA_{g}$. Further, one may hope to describe $A_{g}[m]^{\trop}$ not just as a cell complex, but also as a cone stack in the sense of \cite{cchuw}.

\medskip

\noindent \textbf{Outline.} In Section~\ref{sec:preliminaries}, we define $A_g[m]^\trop$ and discuss preliminaries on isotropic subspaces. In Section~\ref{subsection-tropical-modular-interpretation}, we prove that $A_g[m]^\trop$ admits a tropical modular interpretation, as a parameter space for principally polarized tropical abelian varieties with level structure. We then specialize to the case of abelian surfaces in Sections~\ref{sec:simplicial-complex} and~\ref{sec:homeomorphism-type}. Specifically, we study the link of $A_2[m]^\trop$, giving a simplicial complex description in Section~\ref{sec:simplicial-complex} and a homotopy type description in Section~\ref{sec:homeomorphism-type}.

\medskip

\noindent {\bf Acknowledgments.} 
Part of the content of this paper was originally included in an earlier preprint version of~\cite{ABBCV}. We are again grateful to our colleagues who have provided helpful discussions for that project. Generative AI was not used in preparing this paper.

Initial conversations for this project took place at ICERM in August 2023, during the concurrent programs “Combinatorial Algebraic Geometry: Spring 2021 Reunion Event” and Collaborate@ICERM “Explicit Arithmetic of Shimura Curves.” We thank ICERM for providing a welcoming environment for these interactions.

Eran Assaf was supported by a Simons Collaboration grant (550029, to Voight), and by Simons Foundation Grant (SFI-MPS-Infrastructure-00008651, to Sutherland). Juliette Bruce was partially supported by the National Science Foundation under Award Nos. NSF FRG DMS-2053221 and NSF MSPRF DMS-2002239. Melody Chan was supported by NSF CAREER DMS-1844768, FRG DMS-2053221, and DMS-2401282.

\section{Preliminaries}
\label{sec:preliminaries}

\subsection{Definition of \texorpdfstring{$A_g[m]^\trop$}{Agmtrop}.} \label{subsec:def-Agmtrop}
We introduce the tropicalization of $\cA_g[m]$, as defined in \cite{ABBCV} following \cite{amrt, kkmsd}.  It shall be defined as the geometric realization of the colimit of a certain diagram of rational polyhedral fans.  Fix integers $g \ge 1$ and $m\ge 1$. Let $R$ be a commutative ring.  The relevant examples in this paper are $R=\ZZ$ or $R=\ZZ/m\ZZ$.  Consider the free module $R^{2g}$ with basis \[e_1,\ldots,e_g,f_1,\ldots,f_g\] 
and with the standard symplectic form $J\col R^{2g} \times R^{2g}\to R$ given by
\[
J(e_{i},e_{j})=J(f_{i},f_{j})=0 \quad \quad \text{and} \quad \quad J(e_{i},f_{j})=\delta_{ij}.
\]

\begin{definition} \label{def:isotropic}
    A free $R$-submodule of $R^{2g}$ is an \emph{isotropic subspace} if it is totally isotropic with respect to $J$ and is a direct summand of $R^{2g}$ with free complement.
\end{definition}

\noindent
Note that we use the word ``subspace'' even when $R$ is not a field, e.g., when $R=\ZZ/m\ZZ$ for $m$ composite.

Let \[\Gamma \;=\; \Sp_{2g}(\ZZ) \qquad \text{and} \qquad \Gamma[m] \;=\; \Sp_{2g}(\ZZ)[m] \;=\; \ker \left(\Sp_{2g}(\ZZ) \to \Sp_{2g}(\ZZ/m\ZZ)\right)\] denote the integral symplectic group and its level $m$ principal congruence subgroup, respectively. 

Let $\cW$ be the category whose objects are isotropic subspaces of $\Z^{2g}$, of any rank, with respect to the standard alternating form $J$ on $\Z^{2g}$. For isotropic subspaces $W$ and $W'$, there is an arrow $\gamma\col W \to W'$ in $\cW$ for each $\gamma \in \Gamma[m]$ such that $\gamma \cdot W \subseteq W'$.

Let $W$ be an isotropic subspace of $\Z^{2g}$ with respect to $J$, and let $W_\R = W \otimes_\ZZ \R$.  The real vector space $\Sym^2(W_\R)$ is canonically identified with the space of symmetric bilinear forms on $W_\R^\vee$, and has a distinguished lattice $\Sym^2(W)$ inside it.  This vector space contains the cone of all positive definite forms \[Q\col W_\RR^\vee \times W_\RR^\vee \to \RR, \qquad Q(x,x) > 0 \;\; \text{ for } \;\; x\ne 0. \] This cone shall be denoted $\PD(W^\vee)$ henceforth, dropping the subscript $\RR$ for brevity.  Let $\PD(W^\vee)^{\rt}$ denote the {\em rational closure} of $\PD(W^\vee)$, consisting of the positive semi-definite forms in $ \Sym^2(W_\RR)$ whose kernel is generated by vectors in $W^\vee$.

Admissible collections are defined in \cite{amrt}.  We will not give the full definition, but we recall that an admissible collection gives a functor $\Sigma\col \cW \to \mathsf{RPF}$ to the category of rational polyhedral fans with integral structure.  See~\cite[Sections 1.1-1.3]{ABBCV} for a recollection of admissible collections, along with the definitions of the category $\mathsf{RPF}$ and the association of a functor $\Sigma$ to an admissible collection. Here, we summarize only the facts for $\cA_g[m]$ that are most relevant to this paper. See~\cite[Sections 4.1]{ABBCV} for more details. 

Let $W$ be an object of $\cW$. Consider the level $m$ congruence subgroup 
    $$\GL(W)[m] \; := \; \ker\left( \GL(W) \longrightarrow \GL(W/mW) \right)$$
    of $\GL(W)$. Then $\Sigma(W)$ is a $\GL(W)[m]$-invariant rational polyhedral fan inside the real vector space $\Sym^2(W_\R)$ whose support is $\PD(W^\vee)^{\rt}$.  Its rational polyhedral cones are typically infinite in number but are required to form only finitely many orbits under the $\GL(W)[m]$-action.  The fan $\Sigma(W)$ is an example of an {\em admissible decomposition}~\cite{amrt}.

For $t$ a non-negative integer, we use $\PD_t^{\rt}$ to denote the rational closure of the cone of positive-definite matrices in $\Sym^2(\R^t)$; this is precisely the set of positive semi-definite forms on $\R^t$ whose kernel is defined over $\Q$. A choice of basis $W \cong \Z^t$ gives an isomorphism $\Sym^2(W_\R) \cong \Sym^2(\R^t)$ which identifies $\PD(W^\vee)^{\rt} \cong \PD_t^{\rt}$.  

There are three classical examples of admissible decompositions for $\PDrt_t$ for all $t$, called the \emph{perfect cone} or \emph{first Voronoi}, the \emph{second Voronoi}, and the \emph{central cone decompositions}. Fixing any of these three families of examples, the induced decompositions of $\PD(W^\vee)^{\rt}$ for all $W \in \cW$ form an admissible collection.  All three coincide when $t\le 2$ \cite[Example 4.3.1]{bmv} though not in general. When they coincide, we shall default to referring to the {\em perfect cone} decomposition and call the cones {\em perfect cones}.  The terminology comes from the bijection between top dimensional cones and perfect quadratic forms up to scaling. An explicit description is recalled next; see~\cite[(2.3), (8.8), (8.9), (9.10)]{namikawa} for further details.

\begin{proposition}  \label{prop:perf-2} (The perfect cone decomposition of $\PD_2^{\rt}$) Let $\Sigma(\ZZ^2)$ denote the perfect cone decomposition of $\PDrt_2$.  
 The one-dimensional perfect cones of $\Sigma(\ZZ^2)$ are the rays $\R_{\geq 0} \left\langle vv^t\right\rangle$ as $v$ ranges over all primitive vectors in $\Z^2$.  Thus the rays are in bijection with the set \[\{[v] \in \ZZ^2/\{\pm 1\}\col v \textrm{ primitive in }\ZZ^2\}\] where $[v] = \{v, -v\}$ is the equivalence class of $v$ up to sign.  The vectors $v,-v$ are called the {\em minimal vectors} for this ray. The two- and three-dimensional perfect cones in $\Sigma(\ZZ^2)$ are precisely the polyhedral cones of the form
    \begin{equation} \label{eq:2-and-3-dimensional-perfect-cones}
        \R_{\geq 0}  \left\langle vv^t, \, ww^t\right\rangle \quad \text{ and } \quad \R_{\geq 0}  \left\langle vv^t, \, ww^t, \, (v+w)(v+w)^t \right\rangle,
   \end{equation}
    respectively, for all $v, w \in \Z^2$ forming a $\Z$-basis.

A picture of the \emph{link} $(|\Sigma(\ZZ^2)|\setminus \{0\} ) / \, \R_{>0}$ is shown in Figure~\ref{fig:mod3}. One-, two-, and three-dimensional perfect cones are shown as vertices, edges, and triangles, respectively.

\end{proposition}

    \begin{figure}[h]
\centering
\begin{subfigure}{.5\linewidth}
  \centering
  \includegraphics[width=.8\linewidth]{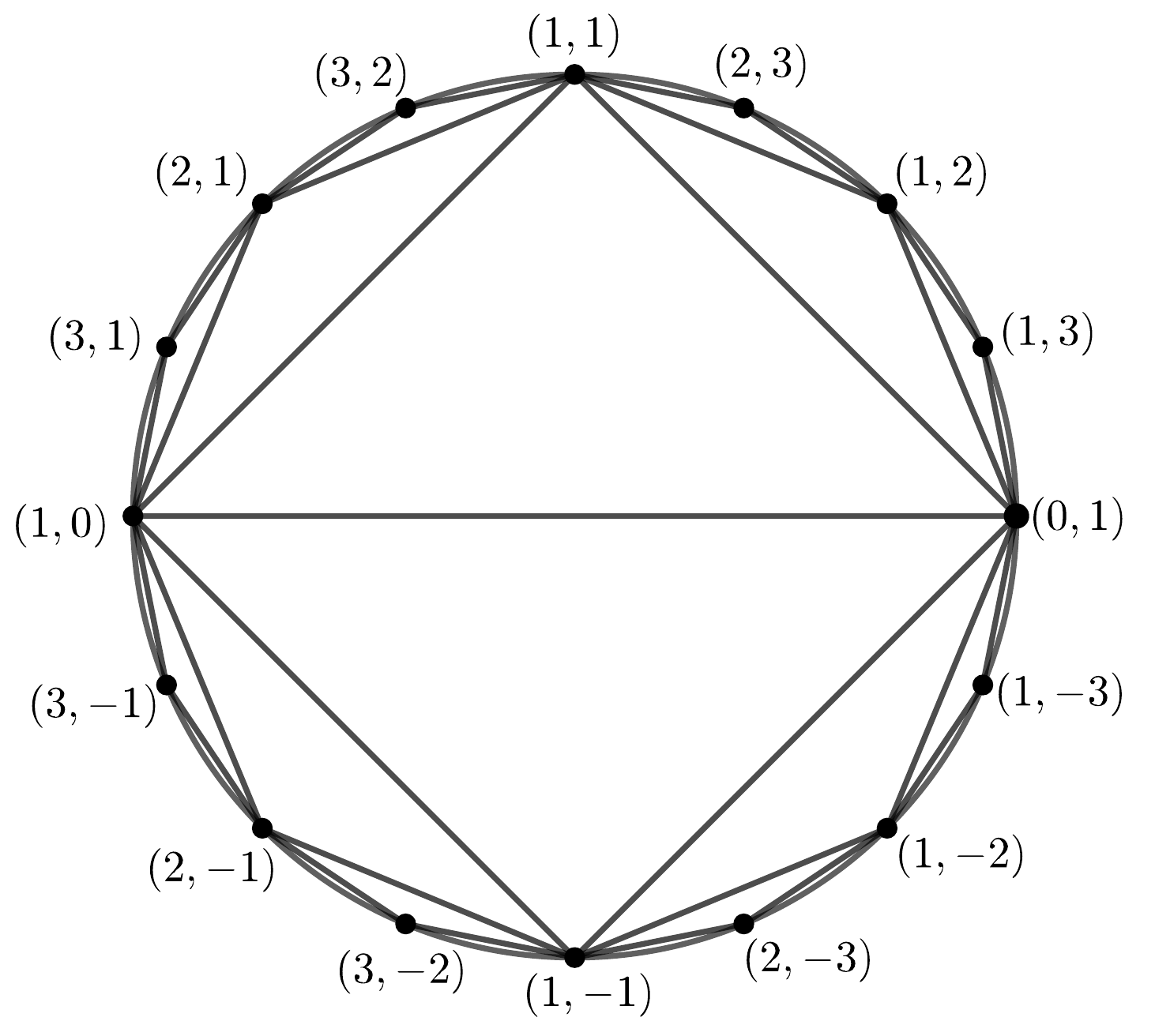}
\end{subfigure}%
\begin{subfigure}{.5\linewidth}
  \centering
  \includegraphics[width=.8\linewidth]{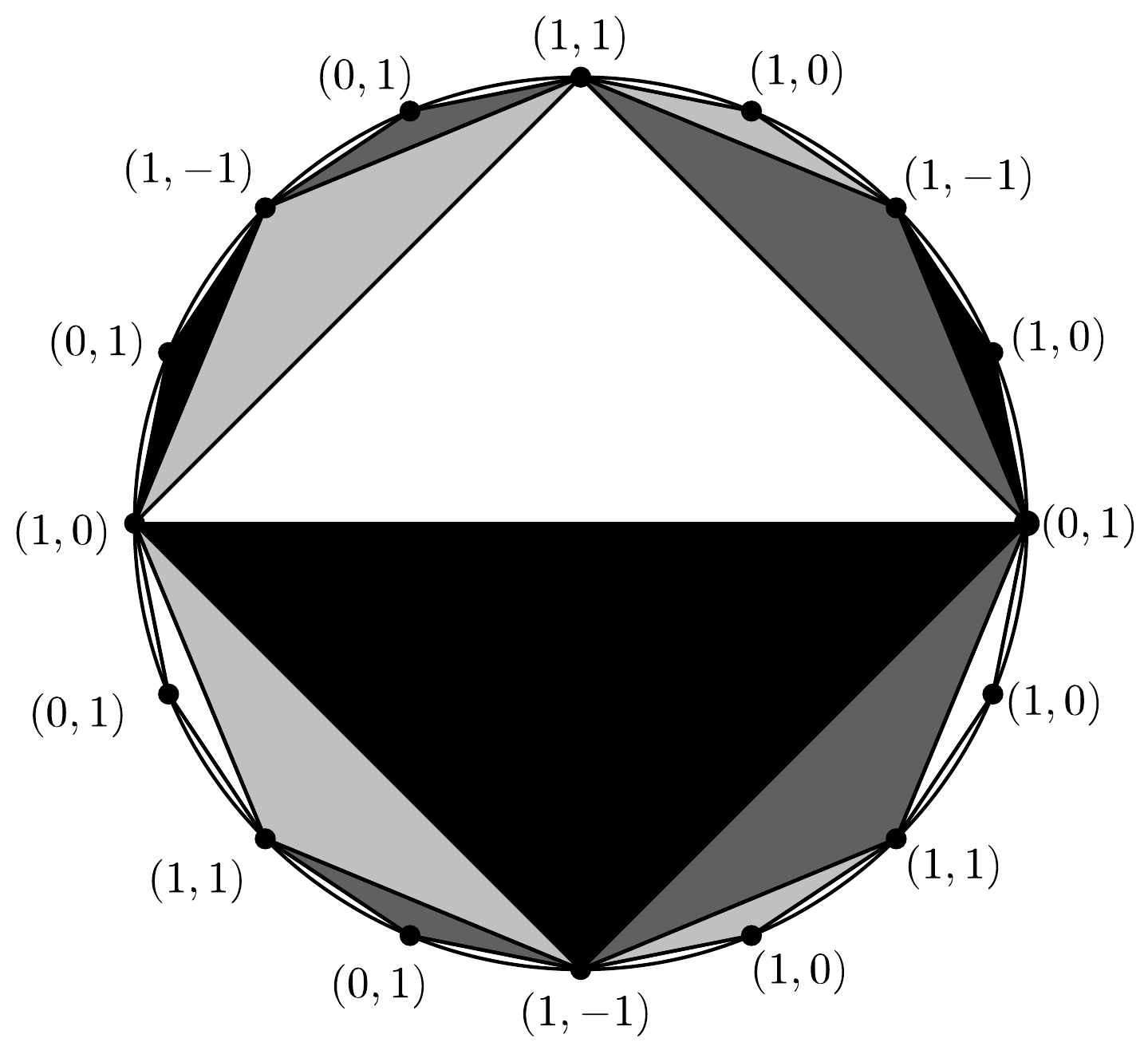}
\end{subfigure}
\caption{Left: the link of the perfect cone admissible decomposition of $\PDrt_2$, with the rays labeled by their corresponding minimal vector (up to sign). Right: under the action of $\GL_2(\Z)[3]$, triangles with like colors and vertices with the same label are glued together.}
\label{fig:mod3}
\end{figure}

\begin{definition}  \label{def:trop-moduli-space}
The {\em tropical moduli space of abelian varieties with level structure $m$}, denoted $A_g[m]^\trop$, is the colimit in $\mathsf{Top}$ of the composition
\begin{equation*}
    \cW \xrightarrow{\quad \Sigma \quad } \mathsf{RPF} \xrightarrow{\quad |\,\cdot\,| \quad } \mathsf{Top},
\end{equation*}
where $\Sigma$ is the functor induced by an admissible collection and the second functor $|\, \cdot \, |$ is geometric realization.
\end{definition}

Thus $A_g[m]^\trop$ is obtained by gluing the geometric realizations of the polyhedral fans $\Sigma(W)$ associated to isotropic subspaces $W \subseteq \Z^{2g}$ along the linear maps induced by elements $\gamma \in \Gamma[m]$ satisfying $\gamma \cdot W \subseteq W'$.  The linear maps glue cones isomorphically to cones.  The homeomorphism type of $A_g[m]^\trop$ is independent of choice of admissible collection \cite[Theorem 1.21]{ABBCV}, so we may speak of a single well-defined space $A_g[m]^\trop$.

\begin{remark} \label{rem:self-arrows}
    The elements of $\Gamma[m]$ that stabilize an isotropic subspace $W$ induce automorphisms of the polyhedral fan $\Sigma(W)$.  In fact, the automorphisms that arise are exactly the ones induced by the action of the level $m$ congruence subgroup $\GL(W)[m]$.

    Precisely,  
    \cite[Proposition 4.1]{ABBCV} shows that 
    the map $\Aut_\cW(W) \to \Aut(\Sigma(W))$ induced by $\Sigma$ factors as
    \[\Aut_\cW(W) \to \GL(W)[m] \to \Aut(\Sigma(W)).\]
Therefore, $A_g[m]^\trop$ is obtained from copies of quotient spaces 
    \begin{equation} \label{eq:quotient-PD-cone}
        \PD(W^\vee)^{\rt}/\GL(W)[m] \;\; \cong \;\; \PD_t^{\rt}/\GL_t(\Z)[m],
    \end{equation}
    where $W\in \cW$ and $t = \rank \, W$, glued according to the action of $\Gamma[m]$ on isotropic subspaces. Here, it is important that $\PD(W^\vee)^{\rt}$ and $\PD_t^{\rt}$ are topologized with their Satake topology rather than with their subspace topology from the Euclidean vector spaces $\Sym^2(W_\RR)$ and $\Sym^2(\RR^t)$; see~\cite[Remark 1.15]{ABBCV}. 
    
    In Section~\ref{sec:homeomorphism-type}, we describe the topology of the spaces~\eqref{eq:quotient-PD-cone} arising when $g = 2$.        See Figure~\ref{fig:mod3} for a depiction of the orbits of the action of $\GL_2(\Z)[3]$ on the perfect cone decomposition of $\PDrt_2$.
\end{remark}

\begin{remark}
    For $m = 1$, Definition~\ref{def:trop-moduli-space} recovers the moduli space $A_g^\trop$ of $g$-dimensional tropical abelian varieties, studied in~\cite{bmv}. In particular, $\cW$ contains a unique isomorphism class of $g$-dimensional isotropic subspaces and $A_g^\trop = \PD_g^{\rt}/\GL_g(\Z)$.
\end{remark}

Although $\Z^{2g}$ contains infinitely many isotropic subspaces, the category $\cW$ has only finitely many isomorphism classes of objects. 
In Remark~\ref{rem:concrete-bijection} below, we give a concrete finite description of these isomorphism classes.

\subsection{Isotropic subspaces} Towards describing the topology of the tropical moduli spaces $A_g[m]^\trop$, 
we study the category $\cW$ by relating isotropic subspaces in $\Z^{2g}$ to isotropic subspaces in $(\Z/m\Z)^{2g}$.
Let 
\[p_m = \begin{cases}
    \frac12 \phi(m) & \textrm{for }m\ge 3,\\
    1 &\textrm{for $m=1,2$.}
\end{cases}\]
Here, $\phi(m)$ denotes Euler's totient function. The significance of $p_m$ is that it is the number of additive generators of $\ZZ/m\ZZ$, up to multiplication by $- 1\in \ZZ/m\ZZ$.  The following is a restatement of~\cite[Proposition 4.8]{ABBCV}.

\begin{Proposition}
\label{prop:boundary-orbits-count} 
Consider the mod $m$ reduction map, which associates to an isotropic subspace $W \subset \Z^{2g}$ an isotropic subspace $\ov{W} \subset (\Z/m\Z)^{2g}$. In every rank $1 \leq t \leq g$, this map descends to a $p_m$--to--$1$ map from the set of $\Gamma[m]$-orbits of isotropic subspaces of $\Z^{2g}$ of rank $t$ onto the set of isotropic subspaces of $(\Z/m\Z)^{2g}$ of rank $t$.
\end{Proposition}

\begin{remark} \label{rem:concrete-bijection}
The isomorphism classes of objects in $\cW$ admit a concrete finite description. An isomorphism $W \to W'$ in $\cW$ is given by an element $\gamma \in \Gamma[m]$ with $\gamma \cdot W = W'$. Thus the isomorphism classes of objects are the $\Gamma[m]$-orbits of integral isotropic subspaces. 
The previous Proposition~\ref{prop:boundary-orbits-count} gives a count of such orbits in every rank. In fact, one can enhance this proposition to an explicit natural bijection
\begin{equation} \label{eq:explicit-bijection}
\begin{tikzcd}[column sep = 3em] \left\{ \begin{matrix} \text{isotropic subspaces}\\ \text{$W \subset \ZZ^{2g}$ of rank $t$} \end{matrix} \right\}/\, \Gamma[m] \arrow[r] & \left\{ (L, \omega) \;\; \bigg| \;\;\begin{matrix} \text{$L \subset (\ZZ/m\ZZ)^{2g}$} \\ \text{ isotropic of rank $t$,} \\ \text{$\omega$ a generator of $\left(\bigwedge^{t} L \right)$} \end{matrix} \right\} / \, (L,\omega)\!\sim\!(L,-\omega)
\end{tikzcd}
\end{equation}
sending $W$ to $(\ov{W}, \ov{w_1} \wedge \cdots \wedge \ov{w_t})$, where $\ov{W}$ is the mod $m$ reduction of $W$, and $w_1,\dots, w_t$ is a $\Z$-basis for $W$. Since $\bigwedge^t L$ is a free $\ZZ / m \ZZ$-module of rank $1$, this bijection explains the appearance of $p_m$ in Proposition~\ref{prop:boundary-orbits-count}.

Nevertheless, passing to a skeleton does not make $\cW$ a finite category as its Hom-sets remain infinite. For example, the automorphism group of the object corresponding to the trivial zero-dimensional isotropic subspace is all of $\Gamma[m]$. In Sections~\ref{sec:simplicial-complex} and~\ref{sec:homeomorphism-type}, we will control these morphisms sufficiently when $g=2$ to give an explicit description of $A_2[m]^\trop$.
\end{remark}

The following lemma, which will be needed later, proves in particular the surjectivity of the map~\eqref{eq:explicit-bijection}. Then, Proposition~\ref{prop:boundary-orbits-count} immediately implies that this map is bijective, as claimed.

\begin{lemma} \label{lemma:isotropic-basis-lifts}
    Let $g,m$ be positive integers. Suppose $a_1, \dots,a_t \in (\Z/m\Z)^{2g}$ form a $(\Z/m\Z)$-basis for a rank $t$ isotropic subspace. Then there exists an isotropic subspace $W \subset \Z^{2g}$ of rank $t$ and a $\Z$-basis $v_1,\dots,v_t$ for $W$ such that $v_i$ reduces to $a_i$ mod $m$ for all $i$.
\end{lemma}

\begin{proof}
    Let $\ov{e_1}, \dots,\ov{e_t}$ be the first $t$ standard symplectic coordinates on $(\Z/m\Z)^{2g}$.
    By Witt's extension theorem over $\Z/m\Z$, an isometry from $\langle a_1,\dots,a_t\rangle$ to $\langle \ov{e_1}, \dots, \ov{e_t}\rangle$ can be extended to an isometry of the whole space $(\Z/m\Z)^{2g}$, so there exists $\gamma \in \Sp_{2g}(\Z/m\Z)$ with $\gamma \cdot a_i = \ov{e_i}$ for all $i$.

    By~\cite[Theorem 1]{newman-smart}, the mod $m$ reduction $\Sp_{2g}(\Z) \to \Sp_{2g}(\Z/m\Z)$ is surjective, so we can pick $\gamma' \in \Sp_{2g}(\Z)$ with reduction $\ov{\gamma'} = \gamma^{-1} \in \Sp_{2g}(\Z/m\Z)$. Setting $v_i = \gamma' \cdot e_i$, we are done.
\end{proof}

\section{Tropical modular interpretation}\label{subsection-tropical-modular-interpretation}

We now give an interpretation of $\Agtropm$ as a tropical moduli space. Namely, we define principally polarized tropical abelian varieties (pptav) of dimension $g$ with level $m$ structure and show that the points of $\Agtropm$ are in bijection with their isomorphism classes.  For $m=1$, this definition agrees with \cite[Definition 4.1.1]{bmv}, as verified below.

Throughout this section, $J$ will denote the standard symplectic form on $\ZZ^{2g}$. If $\Lambda$ is a free abelian group of rank $2g$ and $S \col \Lambda\times\Lambda\to \ZZ$ is a symplectic form (i.e.\ an alternating non-degenerate $\ZZ$-bilinear form), we say that an isomorphism of $\ZZ$-modules $\psi\col \Lambda \to \ZZ^{2g}$ (i.e.\ a choice of $\Z$-basis for $\Lambda$) takes $S$ to the standard symplectic form on $\ZZ^{2g}$ if the diagram 
\[
\begin{tikzcd}[row sep = 3em, column sep = 3em]
    \Lambda \times \Lambda \arrow[r,"S"] \arrow[d,swap, "\psi\times\psi"]& \ZZ \arrow[d,equal]\\
    \ZZ^{2g} \times \ZZ^{2g} \arrow[r,"J"] & \ZZ
\end{tikzcd}
\]
commutes.
In fact, such a choice of $\Z$-basis for $\Lambda$ exists if and only if $\det(S) = 1$.

We likewise refer to the standard symplectic form on $(\ZZ/m\ZZ)^{2g}$ and an isomorphism of free $(\ZZ/m\ZZ)$-modules taking an alternating $(\ZZ/m\ZZ)$-bilinear form to the standard symplectic form. Further, note that if $\Lambda$ is a free abelian group of rank $2g$ and $S \col \Lambda\times\Lambda\to \ZZ$ is an alternating $\ZZ$-bilinear form, there is an induced alternating $(\ZZ/m\ZZ)$-bilinear form on $(\Lambda/m\Lambda)$. Abusing notation slightly, we call this form $S$ as well.

\begin{definition}\label{def-pptav}
A {\em principally polarized tropical abelian variety (pptav) of dimension} $g$ \emph{with level} $m$ \emph{structure} is a quadruple $(\Lambda,S,i,Q)$, where:
\enumnow{\item $\Lambda$ is a free abelian group of rank $2g$;
\item $S\col \Lambda\times\Lambda\to \ZZ$ is a symplectic form on $\Lambda$ having determinant $1$;
\item $i\col \Lambda/m\Lambda\to(\ZZ/m\ZZ)^{2g}$ is an isomorphism taking $S$ to the standard symplectic form on $(\ZZ/m\ZZ)^{2g}$;
\item $Q\col \Lambda_\RR\times\Lambda_\RR\to\RR$ is a positive semi-definite form on $\Lambda_\RR$ with kernel $W_\R^\perp$, for $W \subset \Lambda$ an isotropic subspace with respect to $S$. 
}
\end{definition}

If $(\Lambda,S,i,Q)$ and $(\Lambda',S',i',Q')$ are pptav's of dimension $g$ with level $m$ structure, an \emph{isomorphism} between them is an isomorphism of $\ZZ$-modules $\Lambda\to \Lambda'$ taking $S$ to $S'$, $i$ to $i'$, and $Q$ to $Q'$. 
We denote the \emph{isomorphism class} of $(\Lambda,S,i,Q)$ by $[\Lambda,S,i,Q]$. 

Let $S_{g,m}$ denote the set of isomorphism classes of pptav's of dimension $g$ with level $m$ structure.

\begin{lemma}\label{lem:fix-trop-mod}
    Let $\Lambda = \Z^{2g}$ and $J$ be the standard symplectic form on $\Lambda$. Then there is a surjection of sets:
    \[\Phi \; \col \; \bigsqcup_{W} \, \PD(W^\vee) \; \longrightarrow \; S_{g,m},\] where $W$ ranges over all isotropic subspaces of $\Lambda$.
\end{lemma}

\begin{proof}
    Throughout the proof, $i: \Lambda/m\Lambda \to (\Z/m\Z)^{2g}$ will denote the identity isomorphism. 
    
    Let us first define $\Phi$. For any isotropic subspace $W\subset \Lambda$, the form $J$ induces an isomorphism $W_\R^\vee \cong \Lambda_\R/W_\R^\perp$.  
    By precomposing with the quotient map $\Lambda_\R \twoheadrightarrow \Lambda_\R/W_\R^\perp$, a form $T \in \PD(W^\vee)$ therefore induces a positive semi-definite form on $\Lambda_\R$ with kernel $W_\R^\perp$; let us denote this positive semi-definite form by $Q_T$.
    We define $\Phi$ to send 
    $T$ to the isomorphism class $[\Lambda, J, i, Q_T]$.

    To prove surjectivity, note first that (the isomorphism class of) any pptav of the form $(\Lambda, J, i, Q)$ is in the image of $\Phi$. Indeed, given any positive semi-definite form $Q$ on $\Lambda_\R$ with kernel $W_\R^\perp$ for some isotropic subspace $W \subset \Lambda$, the form $Q$ naturally induces a form $T_Q \in  \PD(W^\vee)$ and the image of $T_Q$ under $\Phi$ precisely equals $[\Lambda, J, i, Q]$.

    So we only need to show that any pptav $(\Lambda', J',i',Q')$ of dimension $g$ with level $m$ structure is isomorphic to some pptav of the form $(\Lambda, J, i, Q)$. Let $I' : \Lambda' \to \Z^{2g}$ be an isomorphism taking $J'$ to the standard symplectic form on $\Z^{2g}$; such an isomorphism exists because $J'$ is assumed to have determinant $1$ by Definition~\ref{def-pptav}. Moreover, since the mod $m$ reduction $\Sp_{2g}(\Z) \to \Sp_{2g}(\Z/m\Z)$ is surjective by \cite[Theorem 1]{newman-smart}, we can pick $I'$ so that the diagram
    \[
\begin{tikzcd}[column sep = 2.5em, row sep = 2.5em]
    \Lambda' \arrow[r, "I'"] \arrow[d,two heads]& \ZZ^{2g} = \Lambda \arrow[d,two heads] \\
    \Lambda'/m\Lambda' \arrow[r,"i'"] & (\ZZ/m\ZZ)^{2g} = \Lambda/m\Lambda
\end{tikzcd}
\]
commutes. Therefore, the map $I': \Lambda' \to \Lambda$ induces an isomorphism of pptav's between $(\Lambda',J',i',Q')$ and $(\Lambda,J,i,Q)$, for $Q$ defined as the pullback of $Q'$ along $(I')^{-1}$.
\end{proof}

\begin{theorem}\label{thm:modular interpretation}
    There is a bijection of sets $\Psi: \Agtropm \to S_{g,m}$
    such that $\Psi \circ \pi = \Phi$, where $\Phi$ is the map from Lemma \ref{lem:fix-trop-mod} and $\pi :  \bigsqcup_{W} \PD(W^\vee) \to \Agtropm$ denotes the natural surjection.
\end{theorem}

\begin{proof}
    Let us construct inverse maps $\Psi:\Agtropm\to S_{g,m}$ and $\Psi':S_{g,m}\to \Agtropm$. Continue to let $i:\Lambda/m\Lambda \to (\Z/m\Z)^{2g}$ denote the identity isomorphism.
    
    First, as noted in the theorem statement, the set $\bigsqcup_W \PD(W^\vee)$ surjects onto $\Agtropm$ because
    $$
    \mathrm{supp}\left(\PD(W^\vee)^{\rt} \right) =
    \bigsqcup_{W' \subset W} \PD((W')^\vee).
    $$

    Given a point $p \in \Agtropm$, let $T\in \PD(W^\vee)$ be an element covering it. We define $\Psi(p)=[\Lambda, J,i,Q_T]$ for $Q_T:\Lambda_\R \times \Lambda_\R \to \R$ the positive semi-definite form with kernel $W_\R^\perp$ that $T$ induces as in the proof of Lemma~\ref{lem:fix-trop-mod}. This is well-defined because two forms $T \in \PD(W^\vee)$ and $T' \in \PD\big((W')^\vee \big)$, for $W,W' \subset \Lambda$ isotropic, descend to the same point in $\Agtropm$ if and only if the induced forms $Q_T, Q_{T'}$ on $\Lambda_\R$ differ by an element of $\Gamma[m] = \Sp_{2g}(\Z)[m]$.

    Conversely, recall from the proof of Lemma~\ref{lem:fix-trop-mod} that every isomorphism class in $S_{g,m}$ has a representative $(\Lambda,J,i,Q)$. Define $\Psi'$ to send $[\Lambda,J,i,Q]$ to the point $\pi(T_Q) \in \Agtropm$, where $T_Q$ is as described in the proof of Lemma~\ref{lem:fix-trop-mod}. This is well-defined because, if $[\Lambda, J, i, Q] = [\Lambda, J, i, Q']$, then there exists an element of $\Sp_{2g}(\Z)[m]$ taking $Q$ to $Q'$, meaning that $\pi(T_Q) = \pi(T_{Q'})$.

It is straightforward to check that $\Psi$ and $\Psi'$ are inverses of each other.   
\end{proof}

\begin{remark}\label{rem:same as bmv}
    For $m=1$, i.e.\ the case with no level structure, pptav's were defined by \cite{bmv} as pairs $(L, Q)$, where $L\cong \ZZ^g$ is a free abelian group and $Q$ is a positive semi-definite form on $L_\RR$ with $L$-rational kernel; and two such pairs are isomorphic if they differ by an element of $\GL_g(\Z)$. For this $m=1$ case, there is a natural bijection between isomorphism classes of pptav's as in Definition~\ref{def-pptav} and isomorphism classes of pptav's in the sense of~\cite{bmv}. Indeed, this is precisely Theorem~\ref{thm:modular interpretation} specialized to $m = 1$, since $A_g^\trop \cong \PDrt_g\!/\GL_g(\Z)$.
\end{remark}

\section{A simplicial complex structure on \texorpdfstring{$A_2[m]^\trop$}{A2 trop[m]}} \label{sec:simplicial-complex}

Fix $g=2$. We study the simplicial structure underlying the spaces $A_2[m]^\trop$ equipped with the canonical choice (for $g=2$) of admissible collection.  We derive the weight zero compactly supported cohomology of $\mathcal{A}_2[m]$ as a consequence.

\subsection{Gluing pattern} \label{sec:gluing-pattern}

Fix $g=2$, so that $\Gamma[m] = \Sp_{4}(\Z)[m]$ throughout, and fix $\Sigma$ to be the canonical (perfect cone) admissible collection.   
By its definition as a colimit (Definition~\ref{def:trop-moduli-space}), the space $A_2[m]^\trop$ is obtained from the disjoint union
 \[\bigsqcup_{P} |\Sigma(P)|,\]
where $P$ runs over all isotropic subspaces of $\ZZ^{4}$, by quotienting along all maps
\begin{equation}\gamma\col |\Sigma(P')| \to |\Sigma(P)|\end{equation} induced by inclusions $\gamma \cdot P'\subseteq P$, as $\gamma$ ranges over the group $\Gamma[m]$.  

Each $\Sigma(P)$ is a polyhedral fan whose support is $\PD(P^\vee)^{\rt}$.  If $P$ has rank $2$, then $\Sigma(P)$ is as described in Proposition~\ref{prop:perf-2}; if $P$ has rank $1$ then $\Sigma(P)$ is a single ray and if $P=\{0\}$ then $\Sigma(P)$ is a point.   

We now label all the rays in the fans $\Sigma(P)$, as $P$ ranges over all isotropic subspaces of $\ZZ^4$, with labels drawn from the set $(\ZZ/m\ZZ)^4/\{\pm 1\}$.  For any $P$ and any ray $\rho$ of $\Sigma(P)$, the lattice point generator of $\rho$ is of the form $vv^t$ for some primitive $v \in P$, where $v$ is uniquely determined up to sign.  This assertion is straightforward when $\dim P = 1$ and is justified by Proposition~\ref{prop:perf-2} when $\dim P = 2$. 

\begin{notation}
    With the notation above, we label the ray $\rho$ by $\ov{v} \in (\ZZ/m\ZZ)^4/\{\pm 1\}$.
\end{notation}

\begin{proposition}\label{prop:Gamma-m-preserves-labels}
Let $P,P' \subset \Z^4$ be isotropic subspaces; we specifically allow $P = P'$. Consider two perfect cones $\sigma\in \Sigma(P)$ and $\sigma'\in \Sigma(P')$.  If $\alpha$ is a bijection from the set of rays of $\sigma'$ to the set of rays of $\sigma$ that preserves their labels, then there exists $\gamma \in \Gamma[m]$ taking $\sigma'$ isomorphically to $\sigma$ such that $\gamma$ induces $\alpha$. Conversely, for any $\gamma\in \Gamma[m]$ taking $\sigma'$ isomorphically to $\sigma$, the induced bijection between the sets of rays of $\sigma'$ and rays of $\sigma$ is label-preserving.
\end{proposition}

Before we proceed to the proof of this proposition, we record the following lemma.

\begin{lemma}\label{lemma:Gamma-m-transitive-on-bases-equal-mod-m}
    Let $g,m$ be positive integers. Let $V,W \subset \Z^{2g}$ be isotropic subspaces having the same rank $t$, and let $v_1,\dots,v_t \in \Z^{2g}$ and $w_1,\dots,w_t \in \Z^{2g}$ be $\Z$-bases for $V$ and $W$, respectively. If $v_i \equiv w_i$ mod $m$ for all $i$, then there exists $\gamma \in \Gamma[m]$ such that $\gamma \cdot v_i = w_i$ for all $i$. 
\end{lemma}
\begin{proof}
    Pick an element $\gamma_1 \in \Sp_{2g}(\Z)$ satisfying $\gamma_1 \cdot v_i = w_i$ for all $i$. Such $\gamma_1$ exists by Witt's extension theorem over $\Z$, which states that any isometry from $V$ to $W$ can be extended to an element of $\Sp_{2g}(\Z)$.
    
    Reducing mod $m$, the element $\ov{\gamma_1} \in \Sp_{2g}(\Z/m\Z)$ satisfies $\ov{\gamma_1} \cdot \ov{v_i} = \ov{w_i} = \ov{v_i}$ for all $i$. We will next show there exists $\gamma_2 \in \Sp_{2g}(\Z)$ such that $\gamma_2 \cdot v_i = v_i$ for all $i$ and $\ov{\gamma_2} = \ov{\gamma_1}^{-1} \in \Sp_{2g}(\Z/m\Z)$. Once we have this element, we can set $\gamma = \gamma_1 \cdot \gamma_2 \in \Gamma[m]$ and we are done.

    We construct $\gamma_2$ via an appropriate lift of $\ov{\gamma_1}^{-1}$ to $\Sp_{2g}(\Z)$. We can extend $v_1,\dots,v_t \in \Z^{2g}$ to a symplectic basis, so we can assume without loss of generality that $v_i = e_i$ in standard symplectic coordinates on $\Z^{2g}$. 
    The stabilizer of an isotropic space is the semidirect product of its Levi component and its unipotent radical \cite[Theorem III.3.10]{amrt}.
    We apply this general fact to the element $\ov{\gamma_1}^{-1} \in \Sp_{2g}(\Z/m\Z)$, which fixes $\ov{v_i} = \ov{e_i}$, and obtain the following product decomposition:
    $$\ov{\gamma_1}^{-1} \; = \; 
    \begin{pmatrix} 
    \id & 0 & 0 & 0 \\ 
    0 & a & 0 & b\\
    0 & 0 & \id & 0\\
    0 & c & 0 & d
    \end{pmatrix} \cdot \begin{pmatrix} 
    \id & e & f & h \\ 
    0 & \id & {}^th & 0\\
    0 & 0 & \id & 0\\
    0 & 0 & -{}^te & \id
    \end{pmatrix}.$$
    Here, the two matrices are written in block form, with the sizes of the vertical/horizontal blocks being $t, g-t, t, g-t$. Each entry denotes a matrix with entries in $\Z/m\Z$ and, furthermore, we know that $\left(\begin{smallmatrix} a & b \\ c & d \end{smallmatrix} \right) \in \Sp_{2(g-t)}(\Z/m\Z)$ and $f + h \cdot {}^te \in \Mat_{t}(\Z/m\Z)$ is symmetric. 
    
    By~\cite[Theorem 1]{newman-smart}, the mod $m$ reduction map $\Sp_{2(g-t)}(\Z) \to \Sp_{2(g-t)}(\Z/m\Z)$ is surjective, so there exists a lift $\left(\begin{smallmatrix} A & B \\ C & D \end{smallmatrix} \right) \in \Sp_{2(g-t)}(\Z)$ of $\left(\begin{smallmatrix} a & b \\ c & d \end{smallmatrix} \right)$. We also lift the blocks $e,f,h$ to integer matrices $E,F,H$ that satisfy the analogous condition that $F + H \cdot {}^tE \in \Mat_{t}(\Z)$ is symmetric. We can guarantee this condition holds as follows: lift the blocks $e,h$ to integer matrices $E,H$ by lifting each entry freely; then, lift the symmetric matrix $f + h \cdot {}^te$ to a symmetric matrix $M \in \Mat_{t}(\Z)$; finally, set $F = M - H \cdot {}^tE$. We obtain the matrix
    $$\gamma_2 \; := \; \begin{pmatrix} 
    \id & 0 & 0 & 0 \\ 
    0 & A & 0 & B\\
    0 & 0 & \id & 0\\
    0 & C & 0 & D
    \end{pmatrix} \cdot \begin{pmatrix} 
    \id & E & F & H \\ 
    0 & \id & {}^tH & 0\\
    0 & 0 & \id & 0\\
    0 & 0 & -{}^tE & \id
    \end{pmatrix} \; \in \; \Sp_{2g}(\Z)$$
    having the desired properties.
\end{proof}

\begin{proof}[Proof of Proposition~\ref{prop:Gamma-m-preserves-labels}]
The second statement follows from the fact that the action of $\Gamma[m]$ on $\ZZ^4$ leaves reductions mod $m$ invariant.  

    Now we prove the first statement. First, if $\sigma'$ and $\sigma$ each have at most two rays, then the claim follows from Lemma~\ref{lemma:Gamma-m-transitive-on-bases-equal-mod-m}.  So suppose then that $\sigma'$ and $\sigma$ are both $3$-dimensional, with ray generators given by
    \[v_1'v_1'^t, \,v_2'v_2'^t,\,v_3'v_3'^t\quad\text{and}\quad v_1v_1^t, \,v_2v_2^t,\,v_3v_3^t\]
    respectively, for $v_i'\in P'$ and $v_i \in P$. We may assume that the ray bijection $\alpha$ takes $\R_{\geq 0} \langle v_i'v_i'^t\rangle$ to $\R_{\geq 0} \langle v_iv_i^t\rangle$ for $i=1,2,3$.

By Proposition~\ref{prop:perf-2}, we can assume after changing the signs of the vectors that
    \begin{equation}\label{eq:ray-is-sum-of-other-2}
        v_3' \; = \; v_1' \; + \; v_2'.
    \end{equation}
    
    Moreover, since $\alpha$ is label-preserving, after changing the signs of the generators $v_1,v_2,v_3$ if necessary, we get \begin{equation}\label{eq:generators-equal-mod-m}v_i' \equiv v_i \mod m, \quad \; \text{for } \; i=1,2,3.\end{equation}
    Proposition~\ref{prop:perf-2} implies that 
    $v_3 = \pm v_1 \pm v_2$.
In fact, we can assume     \begin{equation}\label{eq:ray-is-sum-of-other-2-duplicate}
        v_3 \; = \; v_1 \; + \; v_2.
    \end{equation} without changing either of the equalities~\eqref{eq:ray-is-sum-of-other-2} or~\eqref{eq:generators-equal-mod-m}. Indeed, if $m = 1, 2$, we can change the signs of $v_1$ and $v_2$ freely without changing~\eqref{eq:generators-equal-mod-m}. If $m > 2$, then equations~\eqref{eq:ray-is-sum-of-other-2} and~\eqref{eq:generators-equal-mod-m}, together with the fact that all chosen generators reduce to primitive vectors in $(\Z/m\Z)^4$, imply that $v_3 = v_1 + v_2$ as desired.

Proposition~\ref{prop:perf-2} implies that $v_1',v_2'$ are a $\ZZ$-basis for $P'$ and $v_1,v_2$ are a $\ZZ$-basis for $P$.  By Lemma~\ref{lemma:Gamma-m-transitive-on-bases-equal-mod-m}, there exists $\gamma \in \Gamma[m]$ such that $\gamma(v_i') = v_i$ for $i=1,2$.  Then~\eqref{eq:ray-is-sum-of-other-2} and~\eqref{eq:ray-is-sum-of-other-2-duplicate} imply that $\gamma(v_3') = v_3$ also. 
\end{proof}

\begin{remark}\label{rem:honest-cc}
    For $m = 1$, all labels are identical since $(\Z/1 \Z)^4$ is trivial. For $m \geq 2$, the rays of any perfect cone have pairwise distinct labels. This follows, for example, from the description of the perfect cones in Proposition~\ref{prop:perf-2}. 
    As a consequence, the space $A_2[m]^{\trop}$ is an honest cone complex and not merely a generalized cone complex if and only if $m\ge 2$. That is, if $m\ge 2$, then no cone is glued to itself via non-trivial automorphisms.
\end{remark}

\subsection{Simplicial complex structure on the link} 

We continue to fix $g=2$ and fix the perfect cone decomposition.  Every perfect cone is simplicial when $g\le 2$. Therefore, for $m \geq 2$, the {\em link}
\[LA_2[m]^\trop := (A_2[m]^\trop \setminus\{0\}) / \, \R_{>0}\]
is a simplicial complex.  Using the ray labeling in the previous section and its properties proved in Proposition~\ref{prop:Gamma-m-preserves-labels}, we give an explicit finite description of this simplicial complex, which is one of the main theorems of this paper.

Say that an element $(v_1,\dots,v_n) \in (\Z/m\Z)^n$ is \emph{primitive} if the ideal generated by $v_1,\dots,v_n$ is the whole ring $\Z/m\Z$. Equivalently, an element in $(\ZZ/m\ZZ)^{n}$ is primitive if it generates a subspace of $(\ZZ/m\ZZ)^{n}$ of dimension $1$.

\begin{theorem} \label{thm:simplicial-complex-description}
Fix $m \geq 2$. The link $LA_2[m]^\trop$ is isomorphic, via the ray labeling from Section~\ref{sec:gluing-pattern}, to the $2$-dimensional simplicial complex $\Delta_2[m]$ defined as follows.
\begin{enumerate}
    \item[(a)] The vertex set is the set of primitive elements in $(\Z/m\Z)^4/\{\pm 1\}$.
    \item[(b)] For $v \in (\ZZ/m\ZZ)^4$, write $[v]$ for its equivalence class up to sign. Then two vertices $[v]$ and $[w]$ form an edge if and only if $\{v,w\}$ is a $(\ZZ/m\ZZ)$-basis for an isotropic plane in $(\Z/m\Z)^4$.
    \item[(c)] Three vertices $[u],$ $[v]$, and $[w]$ form a 2-simplex if and only if each pair forms an edge and there exists a choice of signs such that \[\pm u \pm v \pm w = 0.\]
\end{enumerate}
\end{theorem}

\begin{proof}
The description of perfect cones of rank $2$ in Proposition~\ref{prop:perf-2} gives a natural map of simplicial complexes 
$\bigsqcup_P L\Sigma(P) \to \Delta_2[m]$.  The two halves of Proposition~\ref{prop:Gamma-m-preserves-labels} show that this map descends to a well-defined and injective map of simplicial complexes \[LA_2[m]^\trop \to \Delta_2[m],\]
which, we now claim, is also surjective.

It suffices to prove that the map is surjective onto the $2$-simplices of $\Delta_2[m]$.  Suppose $\sigma$ is a $2$-simplex in $\Delta_2[m]$, with vertices $[u]$, $[v]$, and $[w]$, for some $u,v,w\in (\ZZ/m\ZZ)^4$. Changing signs if necessary, we may assume that $w = u+v$.  Since $u$ and $v$ are a $\ZZ/m\ZZ$-basis for an isotropic plane, it follows from Lemma~\ref{lemma:isotropic-basis-lifts} that they have lifts $\tilde u$ and $\tilde v \in \ZZ^4$ that are a $\ZZ$-basis for an isotropic plane $P$ of $\ZZ^4$.  Then the perfect cone in $\Sigma(P)$ with rays $\tilde{u}\tilde{u}^t, \tilde{v}\tilde{v}^t, (\tilde{u}+\tilde{v})(\tilde{u}+\tilde{v})^t$ is sent to $\sigma$, showing surjectivity.
\end{proof}

\begin{remark} \label{rem:edge-contained-in-triangles}
    For later use, we record that when $m > 2$, every edge in $LA_2[m]^\trop$ is contained in exactly two $2$-simplices.  Indeed, we established $LA_2[m]^\trop \cong \Delta_2[m]$, and in $\Delta_2[m]$, the edge spanned  by $[u]$ and $[v]$, for $u,v\in (\ZZ/m\ZZ)^4$, is contained in exactly two $2$-simplices, having vertices labeled 
    $$[u], \; [v], \; [u+v] \quad \text{ and } \quad [u], \; [v], \; [u-v]\,$$
    respectively.  If instead $m = 2$, then $u+v = u-v$, so  there is a unique 2-simplex in $LA_2[m]^\trop$ containing each edge.  
\end{remark}

\subsection{The case \texorpdfstring{$m = 2$}{m2}}

We use Theorem~\ref{thm:simplicial-complex-description} to give a concrete cellular description of the link of $A_2[2]^\trop$ and deduce the weight zero cohomology of $\cA_2[2]$ in Corollary~\ref{cor:A22}.

\begin{corollary} \label{cor:link-A22}
    The link of $A_2[2]^\trop$ is a simplicial complex of pure dimension $2$, having $15$ vertices, $45$ edges, and $15$ triangles, drawn in Figure~\ref{fig:a22}.
\end{corollary}
\begin{proof}
    By Theorem~\ref{thm:simplicial-complex-description}(a), the vertices are in bijection with the non-zero elements of $(\Z/2\Z)^4$. Moreover, $(\Z/2\Z)^4$ contains $15$ isotropic planes, enumerated in Figure~\ref{fig:a22} (right). It then follows from Theorem~\ref{thm:simplicial-complex-description}(c) that three vertices form a triangle precisely when the corresponding elements of $(\Z/2\Z)^4$ are the three non-zero vectors in an isotropic plane.
\end{proof}

\begin{figure}[h]
\centering
\begin{subfigure}{.53\textwidth}
  \centering
  \includegraphics[width=.9\linewidth]{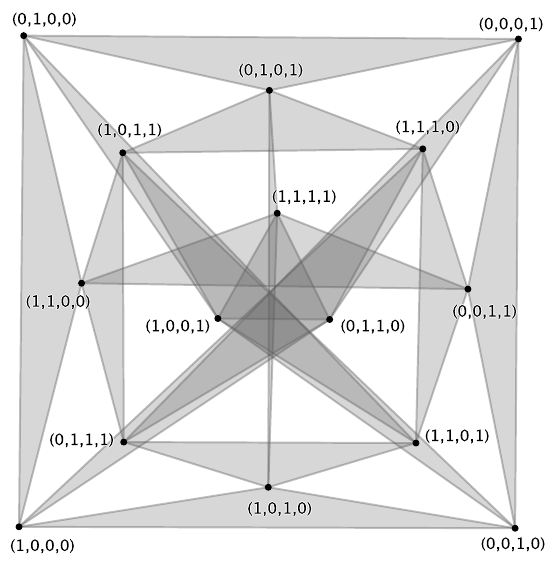}
\end{subfigure}
\begin{subfigure}{.45\textwidth}
  \centering
  \includegraphics[width=.67\linewidth]{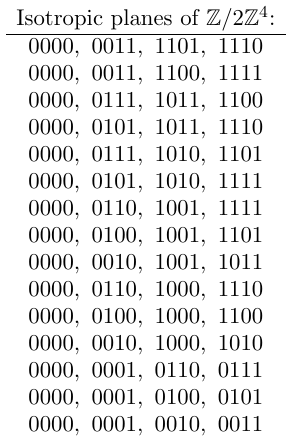}
\end{subfigure}
\caption{Left: the link of $A_2[2]^{\mathrm{trop}}$. Right: the 15 isotropic planes of $(\ZZ/2\ZZ)^4$.} 

\label{fig:a22}
\end{figure}

\begin{corollary}
\label{cor:A22}
    The weight zero compactly supported cohomology of $\cA_2[2]$ is
    $$
    W_0 H^2_c(\cA_2[2]; \QQ) \; \cong \; \mathbb{Q}^{16},
    $$
    and is 0 in all other degrees.
\end{corollary}

\begin{proof}
By \cite[Theorem 1.19]{ABBCV}, there is a canonical isomorphism
\begin{equation}\label{eq:comparison for Agm} 
W_0 H_c^i(\cA_g[m];  \Q) 
\; \cong \; \tilde{H}^{i-1}(LA_g[m]^\trop; \Q),
\end{equation}
for all $i \geq 0$, where $LA_g[m]^\trop$ denotes the link of $A_g[m]^\trop$.
We can deformation retract each triangle from Figure~\ref{fig:a22} into a tripod (i.e., a tripod-shaped graph with $4$ vertices and $3$ edges). The link of $A_2[2]^\trop$ is thus homotopy equivalent to a connected graph with $30$ vertices and $45$ edges, and hence also to a wedge of $16$ circles. The  claim  then follows.
\end{proof}

\subsection{The case \texorpdfstring{$m > 2$}{m bigger than 2}} More generally, using Theorem~\ref{thm:simplicial-complex-description}, we count the simplices in the link of $A_2[m]^\trop$ for $m > 2$.

\begin{lemma}\label{lem:subspacecount}
    Fix $m \geq 2$. In $(\Z/m\Z)^4$, there are 
    $$\frac{m^4}{\phi(m)} \cdot \prod\nolimits_{p \mid m} (1 - p^{-4})$$ isotropic lines, and the same number of isotropic planes. 
    Each isotropic plane contains $$m \cdot \prod\nolimits_{p \mid m} (1 + p^{-1})$$ isotropic lines. 
\end{lemma}
\begin{proof}
    Since every subspace of dimension $1$ is isotropic, the number of isotropic lines equals the number of primitive elements in $(\Z/m\Z)^4$ divided by the number of primitive elements per line. The number of primitive elements is $m^4 \cdot \prod_{p \mid m} (1-p^{-4})$ by a standard count.

    By \cite[Lemma 4.9]{ABBCV}, the number of isotropic planes is
    $$m^{3} \prod\nolimits_{p \mid m}  (1 + p^{-1}) (1 + p^{-2}) \;\; = \;\; \frac{m^4}{\phi(m)} \cdot \prod\nolimits_{p \mid m} (1-p^{-4}).$$
Here, the equality is due to the fact that $\phi(m) = m \prod_{p|m} (1 -p^{-1})$.
    Finally, each isotropic plane contains
    $$\frac{m^2}{\phi(m)} \cdot \prod\nolimits_{p \mid m} (1-p^{-2}) \;\; = \;\; m \cdot \prod\nolimits_{p \mid m} (1 + p^{-1})$$
    lines. Indeed, this is the number of primitive elements in a free $(\Z/m\Z)$-module of rank $2$, which is $m^2 \cdot \prod_{p \mid m} (1-p^{-2})$ by a standard count, divided by the number of primitive elements per line.
\end{proof}

\begin{corollary} \label{cor:counts-simplices-A2m}
    For $m > 2$, the link of $A_2[m]^\trop$ is a simplicial complex of pure dimension $2$, with
    \[\frac{m^4}{2} \cdot \prod\nolimits_{p \mid m} (1 - p^{-4}), \qquad \frac{m^7}{8} \cdot \prod\nolimits_{p \mid m} (1-p^{-2})(1 - p^{-4}), \quad\text{and}\quad \frac{m^7}{12} \cdot \prod\nolimits_{p \mid m} (1-p^{-2})(1 - p^{-4})\]
    simplices of dimension $0$, $1$, and $2$ respectively.
\end{corollary}
\begin{proof}
By Theorem~\ref{thm:simplicial-complex-description}, the number of vertices is half the number of primitive vectors in $(\Z/m\Z)^4$. 
Edges correspond to unordered bases of isotropic planes in $(\Z/m\Z)^4$, up to sign. The number of isotropic planes is given in Lemma~\ref{lem:subspacecount}. The number of unordered bases, up to $\pm 1$, in a free $(\Z/m\Z)$-module of rank $2$ is $\frac{1}{8} \cdot |\GL_2(\Z/m\Z)|$; here, $|\GL_2(\Z/m\Z)| = m^4 \cdot \prod_{p \mid m} (1-p^{-1})(1-p^{-2})$ is a standard count. The number of edges is then obtained as the product of these two quantities. 
Finally, the number of triangles follows because every edge is contained in precisely two triangles, by Remark~\ref{rem:edge-contained-in-triangles}.
\end{proof}

\section{Homotopy type of the link of \texorpdfstring{$A_2[m]^\trop$}{A2m trop}} \label{sec:homeomorphism-type}

In this section, we calculate the homotopy type of $LA_2[m]^\trop$ for $m\ge 3$ and deduce the weight $0$ cohomology $W_0 H^*_c(\cA_2[m];\QQ)$ as a consequence.  These calculations were achieved in Corollary~\ref{cor:A22} in the case $m=2$, where we used the simplicial complex description of $LA_2[2]^\trop$ in Figure~\ref{fig:a22}.

The space $A_2[m]^\trop$ is built out of copies of $\PD_2^{\rt} /\GL_2(\Z)[m]$ glued along common rays, see Remark~\ref{rem:self-arrows} and Equation~\eqref{eq:quotient-PD-cone}.   We next describe the quotient $\PD_2^{\rt} /\GL_2(\Z)[m]$. We shall see that its link is homeomorphic to a closed oriented surface (if $m\ge 3$), which will lead to a computation of the homotopy type of $LA_2[m]^\trop$ as a wedge of surfaces and circles in Theorem~\ref{thm:homotopy}.

\subsection{Quotients of perfect cone decompositions}
\label{sec:PD-quotients}

We describe the quotient $\PD_2^{\rt} / \GL_2(\Z)[m]$ as a cone over a closed orientable surface, whose genus we compute in Lemma~\ref{lem:genus}.

\begin{lemma} \label{lem:glnm} Let $m\ge 3$ and let $t\le n$.
Suppose $v_1,\ldots,v_t$ and $w_1,\ldots,w_t \in \ZZ^n$ are such that \begin{enumerate}
    \item $v_1,\ldots,v_t$, respectively $w_1,\ldots,w_t$, extend to a $\ZZ$-basis
of $\ZZ^n$, and
\item $v_i \equiv w_i\mod m$ for each $i$.
\end{enumerate} 
Then there exists $A\in \GL_n(\ZZ)[m]$ such that $Av_i = w_i$ for $i=1,\ldots,t$.
\end{lemma}
\begin{proof}
By changing the $\ZZ$-basis, we shall assume for convenience that $v_i = e_i$ for each $i$. 
 
Let $R \subset \SL_n(\ZZ)$ be the subgroup of matrices that fix each $e_1,\ldots,e_t$, and similarly define $\ov R \subset \SL_n(\ZZ/m\ZZ)$.  Then the known surjectivity of $\SL_{n-t}(\ZZ) \to \SL_{n-t}(\ZZ/m\ZZ)$, applied to the lower right $(n-t)\times (n-t)$ blocks of $R$ and $\ov R$, shows that the reduction mod $m$ map $R\to \ov R$ is surjective.  

Let $C$ be a matrix whose columns form a $\ZZ$-basis of $\ZZ^n$ extending $w_1,\ldots,w_t$.  We have $\det C = \pm 1$, and $\det C = 1$ if $t=n$ since $\det C$ reduces to $1$ mod $m$.  If $t<n$, after changing the sign of the last column, we may again assume $\det C = 1$.  Then $\ov C \in \ov R$, hence also $\ov {C}{}^{-1}\in \ov R.$  Now use the surjectivity of $R\to \ov R$ to pick $B\in R$ such that $\ov{B} = \ov {C}{}^{-1}$.  Then we claim $A=CB$ has the desired properties. Indeed, for $i=1,\ldots,t$, we have $Ae_i = CBe_i = Ce_i = w_i$.  We also have
\[\ov A = \ov{C}\cdot \ov{B} = \ov{C} \cdot \ov{C}{}^{-1} = 1,\]
showing that $A\in \GL_n(\ZZ)[m]$ as desired.
\end{proof}

\begin{lemma} \label{omegagcount}
     Let $m > 2$. Under the action by $\GL_2(\Z)[m]$, the perfect cone decomposition $\Sigma(\ZZ^2)$ of $\PDrt_2$ has
          \[c_{1,m} := \frac{m^2}{2} \cdot \prod\nolimits_{p \mid m} (1 - p^{-2}), \qquad c_{2,m} := \frac{m^3}{4} \cdot \prod\nolimits_{p \mid m} (1 - p^{-2}), \quad\text{and}\quad c_{3,m} := \frac{m^3}{6} \cdot \prod\nolimits_{p \mid m} (1 - p^{-2})\]
          orbits of cones of dimension $1$, $2$, and $3$, respectively.
\end{lemma}

\begin{proof}
    The explicit description of $\Sigma(\ZZ^2)$ in Proposition~\ref{prop:perf-2} implies that the orbits of rays are in bijection with the set \[\{[v]  :  v \in \ZZ^2 \textrm{ primitive}\}/\GL_2(\ZZ)[m],\]
    where again $[v]$ denotes the equivalence class of $v$ up to sign. Now we claim that the reduction mod $m$ map
    \begin{equation}\label{eq:the-glnm-question}\{[v]  :  v \in \ZZ^2 \textrm{ primitive}\}/\GL_2(\ZZ)[m] \longrightarrow \{[v]: v\in  (\ZZ/m\ZZ)^2 \textrm{ primitive}\}\end{equation}
    is a bijection.  It is an injection by the $t=1$ case of Lemma~\ref{lem:glnm}, and a surjection as follows. If $v=(x,y)\in (\ZZ/m\ZZ)^2$ is primitive, then $\gcd(x,y,m) = 1$.  If $x = 0$, then $(m, y) \in \Z^2$ is a primitive lift of $(x,y)$. Otherwise, pick $t\in \ZZ$ 
    to be the product of all the primes $p$ with $p|x$ and $p\nmid y$.
Then $\gcd(x,y+tm) = 1$ and $(x,y+tm) \in \ZZ^2$ is a primitive lift of $(x,y)$.  This shows~\eqref{eq:the-glnm-question} is a bijection.  Finally, the number of primitive vectors in $(\Z/m\Z)^2$ is $m^2 \cdot \prod_{p \mid m} (1 - p^{-2})$ by a standard count, so the computation of $c_{1,m}$ follows.

Next, Proposition~\ref{prop:perf-2} implies that the $\GL_2(\ZZ)[m]$-orbits of $2$-dimensional cones are in bijection with
\[\{\{[v],[w]\} \col v,w\in \ZZ^2, \, \det (v,w) = \pm 1\}/\GL_2(\ZZ)[m].\]
Again, the reduction mod $m$ map from this set to the set
\[\{\{[v],[w]\} \col v,w\in (\ZZ/m\ZZ)^2, \, \det (v,w) = \pm 1\}\]
is injective by Lemma~\ref{lem:glnm}. It is surjective by the surjectivity of
$\SL^{\pm}_n(\ZZ)\to \SL^{\pm}_n(\ZZ/m\ZZ),$
where $\SL^{\pm}$ denotes matrices of determinant $\pm 1$.  So we have a bijection. To compute the size of the latter set, for every primitive $v\in (\ZZ/m\ZZ)^2/\{\pm1\}$, there are $m$ possible choices of a pair $w\in (\ZZ/m\ZZ)^2/\{\pm1\}$. The computation of $c_{2,m}$ follows.

Finally, we claim that $c_{3,m} = \tfrac23 \cdot c_{2,m}$.  Let $\Sigma(2)$ and $\Sigma(3)$ denote the set of $2$- and $3$-dimensional cones of $\Sigma(\ZZ^2)$, respectively, so that $\Sigma(2)/\GL_2(\ZZ)[m]$ and $\Sigma(3)/\GL_2(\ZZ)[m]$ denote the sets of orbits.  Given $\tau\in \Sigma(2)$, we claim that the two 3-dimensional cones containing $\tau$ as a face are in different $\GL_2(\ZZ)[m]$-orbits.  Indeed, if $\tau$ has rays with ray generators $vv^t$ and $ww^t$ for $v,w\in \ZZ^2$, then the two cones in question have ray generators 
\begin{equation}\label{eq:full-dimensional-perfect-cones}vv^t, ww^t, (v+w)(v+w)^t \quad \text{ and } \quad vv^t, ww^t, (v-w)(v-w)^t.\end{equation}
Furthermore, we have $v+w \not\equiv v-w \mod m$, since $m>2$, so no element of $\GL_2(\ZZ)[m]$ can take one of these 3-dimensional cones to the other.
Similarly, given $\sigma \in \Sigma(3)$, we claim that the three 2-dimensional faces of $\sigma$ are in distinct $\GL_2(\ZZ)[m]$-orbits.  This follows from the fact that for any $\ZZ$-basis $v,w$ of $\ZZ^2$, the three vectors $v,w,v+w$ have distinct reductions mod $m$.

Therefore the incidence correspondence for $\GL_2(\ZZ)[m]$-orbits
\[\{([\tau],[\sigma])\col \tau \in \Sigma(2), \sigma \in \Sigma(3),\tau\prec\sigma \} 
\]
maps 2-to-1 to $\Sigma(2)/\GL_2(\ZZ)[m]$ and 3-to-1 to $\Sigma(3)/\GL_2(\ZZ)[m]$, proving that $c_{3,m} = \tfrac23 \cdot c_{2,m}$.
\end{proof}

\begin{remark}\label{rem:again-simplicial-complex} Analogously to Remark~\ref{rem:honest-cc}, the $\GL_2(\Z)[m]$-action does not induce any non-trivial automorphisms of any cones $\sigma \in \Sigma(\ZZ^2)$ because for all $\ZZ$-bases $v,w$ of $\ZZ^2$, the three vectors $v,$ $w$, and $v+w$ have distinct reductions mod $m$ for $m\ge 2$.  Therefore, for $m\ge 2$, the link of the quotient $\PDrt_2/\GL_2(\Z)[m]$ admits the structure of a simplicial complex.
\end{remark}

\begin{example} \label{ex:A_23trop-and-A_25trop} (The cases $m=3$ and $5$) 
The $\GL_2(\Z)[3]$-action on the perfect cone decomposition of $\PDrt_2$ is depicted in Figure~\ref{fig:mod3} (right). The vertices are labeled by the mod $3$ reductions of the corresponding minimal vectors. After quotienting out by the group action, the link of $\PDrt_2/\GL_2(\Z)[3]$ is a tetrahedron.  Similarly, the link of $\PDrt_2/\GL_2(\Z)[5]$ is an icosahedron, depicted in Figure~\ref{fig:A_2[5]}. The vertices of the icosahedron correspond to $\GL_2(\Z)[5]$-orbits of rays and are labeled by the corresponding elements in $(\Z/5\Z)^2/\{\pm 1\}$.

    \begin{figure}[h]
        \centering
        \includegraphics[width=0.4\linewidth]{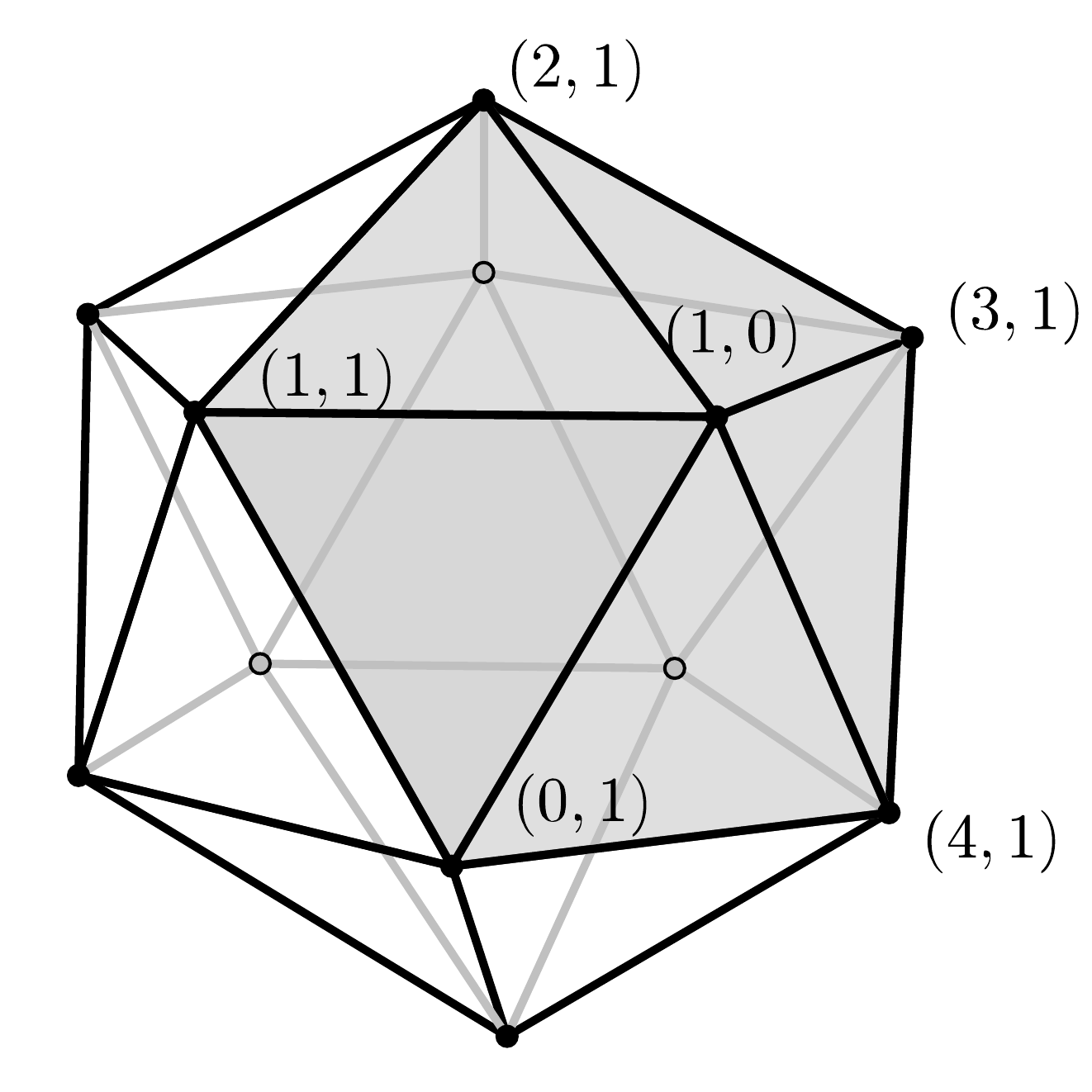}
        \caption{The link of $\PDrt_2/\GL_2(\Z)[5],$ with the vertex labeled $(1,0)$ and its neighbors.} 
        \label{fig:A_2[5]}
    \end{figure}
\end{example}

\begin{lemma}  \label{lem:genus}
    For $m \geq3$, the link of $\PDrt_2/\GL_2(\Z)[m]$ is homeomorphic to a closed orientable surface of genus \[h_m := 1-\frac{1}{2}(c_{1,m} - c_{2,m} + c_{3,m}).\]
\end{lemma}

\begin{remark}\label{rem:modular-curve}
    The link of $\PDrt_2$ is homeomorphic to $\HH\cup \PP^1(\QQ)$ and $\GL_2(\Z)[m] = \SL_2(\Z)[m]$ when $m \geq 3$, so the link of $\PDrt_2/\GL_2(\Z)[m]$ is homeomorphic to the compactified modular curve $X(m)$. 
    Kolay \cite[Proof of Theorem 9.1]{kolay} determines the quotient of the \emph{Farey complex} on $\mathbb{H}$ by $\SL_2(\Z)[m]$, where the Farey complex is defined as follows. Given two rational numbers $a/b$ and $c/d$ (in irreducible form) in $\mathbb{Q} \subset \mathbb{R} \subset \mathbb{C}$, the Farey complex has a semicircle in the upper half plane connecting $a/b$ and $c/d$ if 
    $
    \det \left(
        (a , b),
        (c , d)
    \right ) = \pm 1.
    $
    By the proof of Lemma \ref{omegagcount}, the quotient of the Farey complex on $\mathbb{H}$ by $\SL_2(\Z)[m]$ has the same combinatorial structure as the link of $\PDrt_2/\GL_2(\Z)[m]$.
    For completeness, we include the proof of Lemma \ref{lem:genus} below, but it can also be deduced from loc.\ cit.
\end{remark}

\begin{proof}[Proof of Lemma~\ref{lem:genus}]
The link $L$ of $\PDrt_2/\GL_2(\Z)[m]$ is a $2$-dimensional simplicial complex, as in Remark~\ref{rem:again-simplicial-complex}.  
We verify now that it is a triangulation of an oriented surface.  First, each edge is contained in precisely two triangles; see~\eqref{eq:full-dimensional-perfect-cones}. 
Next, we study the neighborhood of every vertex. Near each ray whose minimal vector reduces to $[v_1] \in (\Z/m\Z)^2/\{\pm1\}$, the triangulated link $L$ looks like an $m$-gon with $[v_1]$ in the center, the $m$ vertices corresponding to 
    \begin{equation*}
        [v_2], \; [v_2 + v_1], \; [v_2 + 2v_1], \ldots , \; [v_2 + (m-1)v_1] \;\; \in  \;\; (\Z/m\Z)^2/\{\pm1\},
    \end{equation*}
    and edges connecting the center to each of the vertices. Here, $v_2$ denotes any vector in $(\Z/m\Z)^2$ satisfying $\det(v_1,v_2) = \pm 1$. An example of such a $5$-gon in the case $m=5$ is shaded in Figure~\ref{fig:A_2[5]}.

    We now check orientability. Away from its finitely many vertices, corresponding to orbits of rays in $\Sigma(\ZZ^2)$, the link inherits an orientation from $\R^3/\R_{>0} \supset L\PD_2$ because the action of $\GL_2(\Z)[m]$ on $\R^3$ is orientation-preserving. This is by the argument in \cite[Lemma 7.1]{elbaz-vincent-gangl-soule-perfect}, since all elements of $\GL_2(\Z)[m] \subset \SL_2(\ZZ)$ have positive determinant when $m\ge 3$. If a smooth surface minus finitely many points is orientable, so is the original surface, so the entire link is orientable.
     Finally, the genus of the surface can be calculated by computing the Euler characteristic using the numbers of vertices, edges, and faces in Lemma~\ref{omegagcount}.
\end{proof}

\subsection{Weight zero cohomology of \texorpdfstring{$\cA_2[m]$}{A2m}}

Using Lemma~\ref{lem:genus}, we describe the homeomorphism type of $LA_2[m]^\trop$ as a collection of orientable surfaces glued along points. This yields a simple description of the homotopy type of $LA_2[m]^\trop$ in Theorem~\ref{thm:homotopy} and the cohomology of $W_0H^*_c(\cA_2[m];\QQ)$ in Theorem~\ref{thm:a2m}.

Recall that $LA_2[m]^\trop$ is isomorphic to the simplicial complex $\Delta_2[m]$ of Theorem~\ref{thm:simplicial-complex-description}.  We will cover $\Delta_2[m]$ with subcomplexes that are isomorphic to the link of $\PD_2^\rat/\GL_2(\ZZ)[m]$ and hence are triangulations of an orientable surface of genus $h_m$. Furthermore, we shall see that these subcomplexes overlap only on vertices.

Let $P \subset (\ZZ/m\ZZ)^4$ be an isotropic subspace of rank 2, and let $[\omega]$ be a generator for the rank 1 free module $\bigwedge^2 P$, up to sign.  Let $S(P,[\omega])$ be the subcomplex of $\Delta_2[m]$ on vertex set
\[\{[v] \in (\ZZ/m\ZZ)^4/\{\pm1\}: v \in P \textrm{ primitive}\}\]
in which two vertices $[v]$ and $[w]$ span an edge of $S(P,[\omega])$ if $v\wedge w = \pm \omega$, and three vertices span a triangle if each pair spans an edge.

\begin{remark}
The way that $S(P,[\omega])$ naturally arises is that if $\widetilde P$ is any lift of $(P,[\omega])$ in the sense of~\eqref{eq:explicit-bijection}, then the canonical map $L\Sigma(\widetilde P)/\GL(\widetilde P)[m] \to \Delta_2[m]$ is an injection with image $S(P,[\omega])$.
\end{remark}

\begin{corollary}\label{cor:link-of-A2tropm} Let $m\ge 3$.  
\begin{enumerate}
    \item Each $S(P,[\omega])$ is isomorphic, as a simplicial complex, to the link of $\PD^\rat_2/\GL_2(\ZZ)[m]$, and hence to a triangulation of a closed oriented surface of genus $h_m$.
    \item Every $1$- and $2$-dimensional simplex of $\Delta_2[m]$ lies in a unique $S(P,[\omega])$.  
\end{enumerate}
In particular, $\Delta_2[m]$, and hence also $LA_2[m]^\trop$, is obtained from 
\[\pi_m \; := \; \frac{m^4}{2} \cdot \prod\nolimits_{p \mid m} (1 - p^{-4})\]surfaces of genus $h_m$, each with $c_{1,m}$ distinct marked points, by gluing these $\pi_m \cdot c_{1,m}$ points into $\pi_m$ equivalence classes. 
\end{corollary}

\begin{proof}
The proof of Lemma~\ref{omegagcount} characterizes $L\PD^\rat_2/\GL_2(\ZZ)[m]$ as isomorphic to the  simplicial complex on the vertex set 
\[\{[v]: v\in (\ZZ/m\ZZ)^2 \textrm{ primitive}\}\]
in which two vertices $[v], [w]$ span an edge if $\det(v,w) = \pm 1$ and three vertices span a triangle if each pair spans an edge.  Any choice of isomorphism $P \cong (\ZZ/m\ZZ)^2$ taking $\omega$ to $\pm 1$ then induces an isomorphism with $S(P,[\omega])$. Now Lemma~\ref{lem:genus} concludes the first statement.  The second statement is immediate: if $[v]$ and $[w]$ span an edge in $\Delta_2[m]$, then they determine both a plane $P = \langle v,w\rangle$ and the generator $v\wedge w\in \bigwedge^2P$.

By Lemma~\ref{lem:subspacecount}, $\pi_m$ equals $p_m = \frac{1}{2}\phi(m)$ times the number of isotropic planes in $(\Z/m\Z)^4$.  Since a free $\ZZ/m\ZZ$-module of rank $1$, for example $\bigwedge^2 P$, has $p_m$ generators up to sign, the count of the surfaces follows.

On a single triangulated surface $S(P,[\omega])$, the number of vertices is $c_{1,m}$, using part (1) and Lemma~\ref{omegagcount}.  Then  after gluing, the equivalence classes of distinguished points correspond to the vertices of $\Delta_2[m]$ and hence to primitive vectors in $(\Z/m\Z)^4$, modulo $\pm 1$, also counted by $\pi_m$.
\end{proof}

\begin{example}
    Continuing Example~\ref{ex:A_23trop-and-A_25trop}, the link of $A_2[3]^\trop$ is given by the boundary of $40$ tetrahedra glued along vertices. The link of $A_2[5]^\trop$ is homeomorphic to the boundary of $312$ icosahedra glued at vertices.
\end{example}

We may explicitly encode the gluing pattern of the surfaces from Corollary~\ref{cor:link-of-A2tropm} in a finite bipartite graph $G_m$ as follows. For $m\geq3$, define
\begin{align*}
\mathcal{V}_{m} &\coloneqq \left\{[v]  \mid \text{$v\in (\ZZ/m\ZZ)^4$ primitive}\right\} \\
\cP_{m} &\coloneqq \left\{(P,[\omega]) \;\bigg|\; \begin{matrix}
    \text{$P\subset (\ZZ/m\ZZ)^4$ an isotropic plane} \\
    \text{$\omega \in \bigwedge^{2} P$ a $(\ZZ/m\ZZ)$-module generator}
\end{matrix}\right\}
\end{align*}
where $[-]$ denotes the equivalence class up to multiplication by $\pm1$. Let $G_{m}$ denote the bipartite graph with vertex classes $\mathcal{V}_{m}$ and $\cP_m$ and one edge joining the vertex $[v]$ to the vertex $(P,[\omega])$ if and only if $v\in P$. 

Let $\Sigma_{h}$ denote the closed orientable surface of genus $h$.

\begin{theorem} \label{thm:homotopy}
    For all $m\geq3$, there is a homotopy equivalence:
    \begin{equation} \label{eq:homotopy-equivalence}
        LA_{2}[m]^{\trop} \;\; \simeq  \;\; \left(\bigvee_{i=1}^{\pi_m} \Sigma_{h_m} \right) \vee \left( \bigvee_{j=1}^{\beta_{1}(G_m)} S^1\right),
    \end{equation}
    where $\beta_1(G_m)=\pi_m(c_{1,m}-2) + 1$ is the first Betti number of the graph $G_m$.
\end{theorem}

\begin{proof}
    The homotopy equivalence~\eqref{eq:homotopy-equivalence} follows from Corollary~\ref{cor:link-of-A2tropm} once we prove $G_{m}$ is connected. Note that $G_{m}$ is the quotient of the incidence graph of non-trivial isotropic subspaces of $\Z^4$ (with vertices corresponding to isotropic lines and planes, and edges given by inclusion of subspaces) by the natural action of $\Gamma[m]$. So it is enough to check that the incidence graph is connected before quotienting by $\Gamma[m]$. To see this, suppose that $\ell_1$ and $\ell_2$ are two lines in $\ZZ^4$. If $\ell_1$ and $\ell_2$ are orthogonal with respect to the pairing $J$, then $\ell_1+\ell_2 \subset \ZZ^4$ is an isotropic plane containing both $\ell_1$ and $\ell_2$. Otherwise, $\ell_1^\perp \cap \ell_2^{\perp}$ has rank $2$, so we may choose a line $\ell_3 \subset \ell_1^\perp \cap \ell_2^{\perp}$. Then, $\ell_1+\ell_3$ and $\ell_2+\ell_3$ are both isotropic planes, the first containing $\ell_1$ and the second containing $\ell_2$. This shows that $G_{m}$ is connected. 
    
    Finally, 
    the expression for $\beta_1(G_m)$ follows because, by Corollary~\ref{cor:link-of-A2tropm}, the graph $G_{m}$ has $c_{1,m}\pi_m$ edges and $2\pi_{m}$ vertices.
\end{proof}

\begin{remark}\label{rem:steinberg}
    When $m$ is prime, the description of $LA_{2}[m]^{\trop}$ from Theorem~\ref{thm:homotopy} gives a relation between the weight zero compactly supported cohomology of $\cA_{2}[m]$ and the $\pm$-oriented symplectic Tits buildings and Steinberg modules as studied in~\cite{capovillaSearle26}. Since $G_{m}$ is homotopic to a wedge of $\beta_{1}(G_m)$ many circles, \eqref{eq:homotopy-equivalence} may also be phrased as
    
     \begin{equation} \label{eq:homotopy-type-graph}
         LA_{2}[m]^{\trop} \;\; \simeq \;\; \left(\bigvee_{i=1}^{\pi_m} \Sigma_{h_m} \right) \vee  \; G_{m}.
     \end{equation}

    As explained in the proof of Theorem~\ref{thm:homotopy}, the graph $G_m$ is simply the quotient of the rational symplectic Tits building $\cT_4^\omega(\Q)$ by the action of $\Gamma[m]$. When $m$ is prime, this quotient is precisely the $\pm$-oriented symplectic Tits building $\cT^{\omega,\pm}_4(\FF_m)$ defined in~\cite[Definition~3.14]{capovillaSearle26}. Furthermore, Capovilla-Searle defines the $\pm$-oriented symplectic Steinberg module to be $\St^{\omega,\pm}_{4}(\FF_m) \coloneqq \widetilde{H}_{1}(\cT_{4}^{\omega,\pm}(\FF_m); \ZZ)$, see Definition~3.19 in loc.\ cit. Therefore, the homotopy equivalence~\eqref{eq:homotopy-type-graph} shows that $(\St_{4}^{\omega,\pm}(\FF_m)\otimes_{\ZZ}\QQ)^\vee$ is a direct summand of $\widetilde{H}^1(LA_2[m]^\trop; \Q) \cong W_{0}H_c^{2}(\mathcal{A}_{2}[m]; \QQ)$.

    \end{remark}

\begin{theorem}
\label{thm:a2m}
    For $m > 2$, the weight zero compactly supported cohomology of $\cA_2[m]$ is
    \begin{align*}
         W_0 H^2_c(\cA_2[m]; \; \QQ) \; \cong \; \Q^{1+ \pi_{m}(c_{2,m} -c_{3,m})},\ \ \ 
     W_0 H^3_c(\cA_2[m]; \; \QQ) \; \cong \; \Q^{\pi_{m}},
    \end{align*}
    and is 0 in all other degrees. 
\end{theorem}

\begin{proof}
We use the isomorphism~\eqref{eq:comparison for Agm} and 
compute the reduced cohomology of the link of $A_{2}^\trop[m]$ directly from the description in Theorem~\ref{thm:homotopy}:
\begin{align*}
    \widetilde{H}^1(LA_2[m]^\trop; \, \QQ) & \; \cong \; \widetilde{H}^1(\Sigma_{h_m}; \, \QQ)^{\oplus \pi_m} \oplus \widetilde{H}^1(S^1; \, \QQ)^{\oplus \beta_{1}(G_m)} 
    \; \cong \;  \QQ^{2h_m\pi_m + \beta_1(G_m)}, \\
    \widetilde{H}^2(LA_2[m]^\trop;\, \QQ) & \; \cong \; \widetilde{H}^2(\Sigma_{h_m}; \, \QQ)^{\oplus \pi_m} \; \cong \; \QQ^{\pi_m}.
\end{align*}
The claim now follows from $h_{m}=1-\frac{1}{2}(c_{1,m}-c_{2,m}+c_{3,m})$ and $\beta_1(G_{m}) = \pi_m(c_{1,m} - 2)+1$.
\end{proof}

\begin{remark} \label{rem:oda-schwermer}
    The weight zero compactly supported cohomology and the top-weight cohomology of $\cA_g[m]$ are Poincar\'e dual to each other. In \cite[Propositions~5.3 and~5.5]{Oda1990MixedHS}, Oda and Schwermer study the weight filtrations on $H^3(\cA_2[m]; \, \Q)$ and $H^4(\cA_2[m]; \, \Q)$. They give expressions for the top-weight pieces
    in line~(2) of each proposition. Theorem~\ref{thm:a2m} completes this calculation and gives 
    \begin{center}
     \begin{tabular}{rl}
         $\dim \Gr_{6}^W H^4(\cA_2[m]; \; \Q) \;$ & $ = \quad 1+ \frac{m^7}{24} \cdot \prod\nolimits_{p \mid m} (1 - p^{-2})(1 - p^{-4})$ \\
        and $\quad \dim \Gr_{6}^W H^3(\cA_2[m]; \; \Q) \;$ & $= \quad \frac{m^4}{2} \cdot \prod\nolimits_{p \mid m} (1 - p^{-4}).$
     \end{tabular} 
          \end{center}
\end{remark}

\bibliographystyle{alpha}
\bibliography{refs}

@unpublished{ABBCV,
	AUTHOR = {Assaf, Eran and Brandt, Madeline and Bruce, Juliette and Chan, Melody and Raluca Vlad},
	TITLE = {Tropicalizations of locally symmetric varieties},
	YEAR = {2025},
	NOTE = {arXiv:2505.10504},
}

@article{acp,
	author = {Abramovich, Dan and Caporaso, Lucia and Payne, Sam},
	fjournal = {Annales Scientifiques de l'\'{E}cole Normale Sup\'{e}rieure. Quatri\`eme S\'{e}rie},
	journal = {Ann. Sci. \'{E}c. Norm. Sup\'{e}r. (4)},
	number = {4},
	pages = {765--809},
	title = {The tropicalization of the moduli space of curves},
	url = {https://doi.org/10.24033/asens.2258},
	volume = {48},
	year = {2015},
}

@article{allcock-corey-payne-tropical,
	author = {Allcock, Daniel and Corey, Daniel and Payne, Sam},
	fjournal = {Bulletin of the London Mathematical Society},
	journal = {Bull. Lond. Math. Soc.},
	number = {1},
	pages = {193--205},
	title = {Tropical moduli spaces as symmetric {$\Delta$}-complexes},
	url = {https://doi.org/10.1112/blms.12570},
	volume = {54},
	year = {2022},
}

@book {amrt,
    AUTHOR = {Ash, Avner and Mumford, David and Rapoport, Michael and Tai, Yung-sheng},
     TITLE = {Smooth compactifications of locally symmetric varieties},
    SERIES = {Cambridge Mathematical Library},
   EDITION = {Second},
      NOTE = {With the collaboration of Peter Scholze},
 PUBLISHER = {Cambridge University Press, Cambridge},
      YEAR = {2010},
     PAGES = {x+230},
       URL = {https://doi.org/10.1017/CBO9780511674693},
}

@article{bbcmmw-top,
	author = {Brandt, Madeline and Bruce, Juliette and Chan, Melody and Melo, Margarida and Moreland, Gwyneth and Wolfe, Corey},
	fjournal = {Geometry \& Topology},
	journal = {Geom. Topol.},
	number = {2},
	pages = {497--538},
	title = {On the top-weight rational cohomology of {$\Cal A_g$}},
	url = {https://doi.org/10.2140/gt.2024.28.497},
	volume = {28},
	year = {2024},
}

@article{bmv,
	author = {Brannetti, Silvia and Melo, Margarida and Viviani, Filippo},
	fjournal = {Advances in Mathematics},
	journal = {Adv. Math.},
	number = {3},
	pages = {2546--2586},
	title = {On the tropical {T}orelli map},
	url = {http://dx.doi.org/10.1016/j.aim.2010.09.011},
	volume = {226},
	year = {2011},
}

@article{cchuw,
	author = {Cavalieri, Renzo and Chan, Melody and Ulirsch, Martin and Wise, Jonathan},
	journal = {Forum of Mathematics, Sigma},
	pages = {e23},
	publisher = {Cambridge University Press},
	title = {A moduli stack of tropical curves},
	volume = {8},
	year = {2020},
}

@unpublished{capovillaSearle26,
      title={On the top-degree cohomology groups of congruence subgroups of $\text{Sp}_{2n}(\mathbb{Z})$}, 
      author={Fabio Capovilla-Searle},
      year={2026},
      note ={arXiv:2605.29010},
      url={https://arxiv.org/abs/2605.29010}, 
}

@article{elbaz-vincent-gangl-soule-perfect,
	author = {Elbaz-Vincent, Philippe and Gangl, Herbert and Soul\'e, Christophe},
	fjournal = {Advances in Mathematics},
	journal = {Adv. Math.},
	pages = {587--624},
	title = {Perfect forms, {K}-theory and the cohomology of modular groups},
	url = {https://doi.org/10.1016/j.aim.2013.06.014},
	volume = {245},
	year = {2013},
}

@book{kkmsd,
	author = {Kempf, George and Knudsen, Finn Faye and Mumford, David and Saint-Donat, Bernard},
	pages = {viii+209},
	publisher = {Springer-Verlag, Berlin-New York},
	series = {Lecture Notes in Mathematics, Vol. 339},
	title = {Toroidal embeddings. {I}},
	year = {1973}
}

@unpublished{kolay,
author = {Kolay, Sudipta},
year = {2019},
title = {Subgroups of the mapping class group of the torus generated by powers of {D}ehn twists},
note = {arXiv:1909.07360}
}

@article{kannan-song-dual,
    AUTHOR = {Kannan, Siddarth and Song, Terry Dekun},
     TITLE = {The dual complex of {$\Cal{M}_{1,n}(\Bbb P^r,d)$} via the
              geometry of the {V}akil-{Z}inger moduli space},
   JOURNAL = {Adv. Math.},
  FJOURNAL = {Advances in Mathematics},
    VOLUME = {492},
      YEAR = {2026},
     PAGES = {Paper No. 110910, 41},
       URL = {https://doi.org/10.1016/j.aim.2026.110910},
}

@article{el-maazouz-helminck-roehrle-souza-yun-topology,
    AUTHOR = {Maazouz, Yassine El and Helminck, Paul Alexander and R\"ohrle,
              Felix and Souza, Pedro and Yun, Claudia He},
     TITLE = {On the topology of the moduli of tropical unramified
              {$p$}-covers},
   JOURNAL = {Selecta Math. (N.S.)},
  FJOURNAL = {Selecta Mathematica. New Series},
    VOLUME = {31},
      YEAR = {2025},
    NUMBER = {1},
     PAGES = {Paper No. 14, 42},
       URL = {https://doi.org/10.1007/s00029-024-01007-4},
}

@incollection {MZ08,
    AUTHOR = {Mikhalkin, Grigory and Zharkov, Ilia},
     TITLE = {Tropical curves, their {J}acobians and theta functions},
 BOOKTITLE = {Curves and abelian varieties},
    SERIES = {Contemp. Math.},
    VOLUME = {465},
     PAGES = {203--230},
 PUBLISHER = {Amer. Math. Soc., Providence, RI},
      YEAR = {2008},
      ISBN = {978-0-8218-4334-5},
   MRCLASS = {14T05 (05C38 14H40 14H42)},
  MRNUMBER = {2457739},
       DOI = {10.1090/conm/465/09104},
       URL = {https://doi.org/10.1090/conm/465/09104},
}

@article {newman-smart,
    AUTHOR = {Newman, Morris and Smart, John Roderick},
     TITLE = {Symplectic modulary groups},
   JOURNAL = {Acta Arith.},
  FJOURNAL = {Polska Akademia Nauk. Instytut Matematyczny. Acta Arithmetica},
    VOLUME = {9},
      YEAR = {1964},
     PAGES = {83--89},
       URL = {https://doi.org/10.4064/aa-9-1-83-89},
}

@article{Oda1990MixedHS,
  title={Mixed {H}odge structures and automorphic forms for {S}iegel modular varieties of degree two},
  author={Takayuki Oda and Joachim Schwermer},
  journal={Mathematische Annalen},
  year={1990},
  volume={286},
  pages={481-509},
  url={https://api.semanticscholar.org/CorpusID:121623037}
}

@book {namikawa,
    AUTHOR = {Namikawa, Yukihiko},
     TITLE = {Toroidal compactification of {S}iegel spaces},
    SERIES = {Lecture Notes in Mathematics},
    VOLUME = {812},
 PUBLISHER = {Springer, Berlin},
      YEAR = {1980},
     PAGES = {viii+162},
}

\end{document}